\documentclass[11pt]{amsart}
\usepackage[usenames,dvipsnames]{color}
\usepackage{amsmath,ifthen}
\usepackage{a4wide}
\usepackage[hmargin=32mm, vmargin=27mm, includefoot, twoside]{geometry}
\usepackage[
bookmarksopen=true]{hyperref}
\usepackage[latin1]{inputenc}  
\usepackage[T1]{fontenc}       
\usepackage[english]{babel}
\usepackage{amssymb, enumerate}
\usepackage{latexsym}
\usepackage{mathrsfs}
\usepackage{xspace}
\usepackage{enumerate, paralist}
\usepackage{graphics, graphicx, epsfig}
\usepackage[all, cmtip]{xy}
\usepackage{verbatim} 
\usepackage{color}
\newcommand{\showcomments}{yes}

\newsavebox{\commentbox}
{\ifthenelse{\equal{\showcomments}{yes}}%
{\footnotemark
    \begin{lrbox}{\commentbox}
    \begin{minipage}[t]{1.00in}\raggedright\sffamily\tiny
    \footnotemark[\arabic{footnote}]}
{\begin{lrbox}{\commentbox}}}%
{\ifthenelse{\equal{\showcomments}{yes}}%
{\end{minipage}\end{lrbox}\marginpar{\usebox{\commentbox}}}
{\end{lrbox}}}

\newtheorem{prop}{Proposition}[section]
\newtheorem{thm}[prop]{Theorem}
\newtheorem*{thm*}{Theorem}

\newtheorem*{addendum*}{Addendum}
\newtheorem{cor}[prop]{Corollary}
\newtheorem{lem}[prop]{Lemma}

\newtheorem{thmintro}{Theorem}

\newtheorem{corintro}[thmintro]{Corollary}

\newtheorem{claim}{Claim}
\newtheorem*{claim*}{Claim}
\newtheorem*{convention*}{Convention}
\newtheorem*{RkRigThm}{Rank Rigidity Theorem}
\newtheorem*{RkRigConj}{Rank Rigidity Conjecture}
\newtheorem*{Double}{Double Skewering Lemma}
\newtheorem*{Flipping}{Flipping Lemma}
\newtheorem*{IrredCrit}{Irreducibility Criterion}

\theoremstyle{definition}
\newtheorem*{defn*}{Definition}

\newtheorem{remark}[prop]{Remark}

\newtheorem*{scholium*}{Scholium}
\theoremstyle{remark}

\newtheorem*{example*}{Example}
\numberwithin{equation}{section}

\newcommand{\vareps}{\varepsilon}

\newcommand{\NN}{\mathbf{N}}

\newcommand{\RR}{\mathbf{R}}

\newcommand{\ZZ}{\mathbf{Z}}

\newcommand{\la}{\langle}
\newcommand{\ra}{\rangle}
\newcommand{\inv}{^{-1}}

\newcommand{\g}{\mathfrak{g}}

\newcommand{\centra}{\mathscr{Z}}

\newcommand{\QH}{\widetilde{\mathrm{QH}}}

\newcommand{\hH}{\hat{\mathcal{H}}}
\renewcommand{\H}{{\mathcal H}}
\newcommand{\K}{{\mathcal K}}
\newcommand{\h}{{\mathfrak h}}
\newcommand{\hh}{\hat{\mathfrak h}}
\renewcommand{\k}{{\mathfrak d}}
\newcommand{\hk}{\hat{\mathfrak d}}
\newcommand{\cat}{{\upshape CAT(0)}\xspace}  
\newcommand{\tangle}[2]
{\angle_\mathrm{T}(#1,#2)}
\newcommand{\aangle}[3]
{\angle_{#1}(#2,#3)}
\newcommand{\cangle}[3]
{\overline{\angle}_{#1}(#2,#3)}

\DeclareMathOperator{\Stab}{Stab}

\DeclareMathOperator{\Isom}{Is}

\newcommand{\bd}{\partial_\infty} 

\def\Aut{\mathop{\mathrm{Aut}}\nolimits}

\def\Ess{\mathop{\mathrm{Ess}}\nolimits}
\def\Hess{\mathop{\mathrm{Hess}}\nolimits}
\def\Triv{\mathop{\mathrm{Triv}}\nolimits}
\begin{document}

\title[Rank rigidity of \cat cube complexes]{Rank rigidity for \cat cube complexes}
\author[Pierre-Emmanuel Caprace]{Pierre-Emmanuel Caprace*}
\address{UCLouvain -- IRMP, Chemin du Cyclotron 2, 1348 Louvain-la-Neuve, Belgium}
\email{pe.caprace@uclouvain.be}
\thanks{* F.R.S.-FNRS research associate. Supported in part by FNRS grant F.4520.11}
\author[Michah Sageev]{Michah Sageev$^\ddagger$}
\address{Dept. of Math., Technion, Haifa 32000, Israel}
\email{sageevm@techunix.technion.ac.il}
\thanks{$^\ddagger$Supported in part by ISF grant \#580/07.}
\date{First draft: March 2010; revised: February 2011}%

\begin{abstract}
We prove that any group acting essentially without a fixed point at infinity on an irreducible finite-dimensional \cat cube complex contains a rank one isometry. This implies that the Rank Rigidity Conjecture holds for \cat cube complexes. We derive a number of other consequences for \cat cube complexes, including a purely geometric proof of the Tits Alternative, an existence result for regular elements in (possibly non-uniform) lattices acting on cube complexes, and a characterization of products of trees in terms of bounded cohomology.
\end{abstract}

\maketitle

\tableofcontents

\section{Introduction}\label{sec:intro}
%

\subsection{Rank one and contracting isometries}

Let $X$ be a complete \cat space. A \textbf{rank one isometry} is a hyperbolic $g \in \Isom(X)$ none of whose axes bounds a flat half-plane. 
{A \textbf{contracting isometry} is a hyperbolic $g \in \Isom(X)$ having some axis $\lambda$ such that  the diameter of the orthogonal projection to $\lambda$ of any ball of $X$ disjoint from $\lambda$ is bounded above. A contracting isometry is always a rank one isometry. The converse is false in general, but it holds if the ambient space $X$ is proper, see \cite[Theorem~5.4]{BestvinaFujiwara}. 
}
If $X$ is Gromov hyperbolic, then every hyperbolic isometry {is contracting}. Even if $X$ is not hyperbolic, a contracting isometry $g$ acts on the visual boundary $\bd X$ with a \emph{North-South dynamics}: $g$ has exactly two fixed points in $\bd X$, respectively called \textbf{attracting} and \textbf{repelling}, and the positive powers of $g$ contract the whole boundary minus the repelling fixed point to the attracting one. This {feature justifies the choice of terminology, and} can be exploited to derive a large number of consequences {from the very existence of contracting isometries}; some of these  will be reviewed below. {To put it short, one can say that, in the presence of a contracting isometry, the space $X$ presents some kind of hyperbolic behaviour. }

The notion of rank one and contracting isometries originates in the celebrated Rank Rigidity Theorem for Hadamard manifolds, due to W.~Ballmann, M.~Brin, K.~Burns, P.~Eberlein, R.~Spatzier  (see \cite{Ballmann} for more information and detailed references). A possible formulation of this result, but not the most general one, is the following. Recall that a Hadamard manifold is a simply connected, complete Riemannian manifold of non-positive curvature, and that such a manifold is said to be irreducible if it does not admit a nontrivial decomposition as a metric product of two manifolds. 

\begin{RkRigThm}[\cite{Ballmann85} and \cite{BurnsSpatzier}]
Let $M$ be a Hadamard manifold and $\Gamma$  be a discrete group acting properly and cocompactly on $M$. If $M$ is irreducible, then either $M$ is a higher rank symmetric space or $\Gamma$ contains a rank one isometry. 
\end{RkRigThm}

It is expected that a similar phenomenon holds for much wider classes of \cat spaces. In particular, W.~Ballmann and S.~Buyalo \cite{BallmannBuyalo} formulate the following.

\begin{RkRigConj}
Let $X$ be a locally compact geodesically complete \cat space  and $\Gamma$  be an infinite discrete group acting properly and cocompactly on $X$. If $X$ is irreducible, then $X$ is a higher rank symmetric space or a Euclidean building of dimension~$\geq 2$, or $\Gamma$ contains a rank one isometry. 
\end{RkRigConj}

Recall that a \cat space is called \textbf{geodesically complete} if every geodesic segment can be prolonged to some (not necessarily unique) bi-infinite geodesic line. Besides the manifold case, the Rank Rigidity Conjecture has been confirmed for piecewise Euclidean cell complexes of dimension~$2$ by W.~Ballmann and M.~Brin (see \cite{BallmannBrin1}), {who also obtained relevant partial results in dimension~$3$ (see \cite{BallmannBrin2})}. It also holds within the class of buildings and Coxeter groups (see \cite{CapraceFujiwara}) as well as for right-angled Artin groups (see \cite{BehrstockCharney}). 

The main result of this paper is the following.

\begin{thmintro}[Rank Rigidity for finite-dimensional \cat cube complexes]\label{thmintro:RankRigidity}
Let $X$ be a finite-dimensional \cat cube complex and $\Gamma \leq \Aut(X)$ be a group acting without fixed point in  $X \cup \bd X$. Then there is a convex $\Gamma$-invariant subcomplex $Y \subseteq X$ such that either $Y$ is a product of two unbounded cube subcomplexes or $\Gamma$ contains an element acting on $Y$ as a contracting isometry. 
\end{thmintro}

A slightly more precise version will be obtained in Theorem~\ref{thm:RankOne} below. In particular, we shall see that if $X$ is locally finite and if the $\Gamma$-action is cocompact (but not necessarily proper), then the conclusions of Theorem~\ref{thmintro:RankRigidity} still hold even if there are $\Gamma$-fixed points at infinity. 

Notice that there is no assumption on geodesic completeness or locally finiteness of $X$ in Theorem~\ref{thmintro:RankRigidity}. Moreover, the $\Gamma$-action need not be proper or cocompact. However, the result does not hold if one does not allow to pass to a convex subcomplex in the conclusion. In order to address this issue, we shall say that a group $\Gamma \leq \Aut(X)$  acts \textbf{essentially} on $X$ if no $\Gamma$-orbit remains in a bounded neighbourhood of a half-space of $X$.  Starting with any group $\Gamma \leq \Aut(X)$, there is a way to construct a $\Gamma$-invariant cube subcomplex $Y \subseteq X$ on which $\Gamma$ acts essentially provided any of the following conditions are satisfied (see Propositions~\ref{prop:pruning} and~\ref{prop:fg}  below):
\begin{itemize}
\item $\Gamma$ has no fixed point at infinity.

\item $\Gamma$ has finitely many orbits of hyperplanes.

\item $\Gamma$ is finitely generated and acts properly discontinuously.
\end{itemize}

In particular, if one assumes that the $\Gamma$-action on $X$ is essential in Theorem~\ref{thmintro:RankRigidity}, then the conclusion holds with $Y=X$. In the precise setting of the Rank Rigidity Conjecture, namely for geodesically complete spaces with a discrete cocompact group action, it turns out that the hypotheses of Theorem~\ref{thmintro:RankRigidity} become redundant: 

\begin{corintro}[Rank Rigidity for geodesically complete \cat cube complexes]\label{cor:GeodComplete}
Let $X$ be a locally compact geodesically complete \cat cube complex  and $\Gamma$  be an infinite discrete group acting properly and cocompactly on $X$. Then $X$ is a product of two geodesically complete unbounded convex subcomplexes or $\Gamma$ contains a rank one isometry. 
\end{corintro}

Without the assumption of geodesic completeness, Corollary~\ref{cor:GeodComplete} stills holds provided one passes to a $\Gamma$-invariant subcomplex on which $\Gamma$ acts essentially; it also holds for non-uniform lattices (see Corollary~\ref{cor:RkRigid} below). 

\subsection{Proof ingredients}\label{sec:outline}

An important elementary fact on which our proof of Rank Rigidity for \cat cube complexes relies is the following tool allowing to \textbf{flip} any essential half-space $\h$, {\emph{i.e.} to map $\h$ to a half-space $\k$ containing properly  the complementary half-space $\h^*$ of $\h$. 

\begin{Flipping}
Let $X$ be a finite-dimensional \cat cube complex and $\Gamma \leq \Aut(X)$ be a group acting essentially without a fixed point at infinity. Then for any half-space $\h$, there is some $\gamma \in \Gamma$ such that $\h^* \subsetneq  \gamma.\h $.
\end{Flipping}

A slightly more precise  version of the Flipping Lemma will be proved in Theorem~\ref{thm:Flipping} below. 

{An isometry $g \in \Isom(X)$ is said to \textbf{skewer} a hyperplane $\hh$ bounding some half-space $\h$ if $g.\h \subsetneq \h$ or $\h \subsetneq g.\h$. If $X$ is finite-dimensional, then, given $\Gamma \leq \Aut(X)$, the $\Gamma$-action is essential if and only if every half-space is skewered by some isometry (see {Proposition \ref{prop:essential} } below). This should be compared to the following} consequence of the Flipping Lemma,  which is an essential and flexible tool to construct `many' hyperbolic elements in $\Gamma$. 

\begin{Double}
Let $X$ be a finite-dimensional \cat cube complex and $\Gamma \leq \Aut(X)$ be a group acting essentially without fixed point at infinity. 

Then for any two half-spaces $\k\subset\h$, there exists $\gamma \in\Gamma$ such that $\gamma \h\subsetneq \k\subset\h$.
\end{Double}

\begin{proof}
By the Flipping Lemma, we can find an element $g$ flipping  $\k$ to obtain $g\h^*\subset g\k*\subset \k\subset\h$. Invoking the Flipping Lemma again, we then find an element $a$ flipping $g\h^*$ and, setting $\gamma = ag$, we obtain $\gamma\h\subset g\h^*\subset\k$, as required. 
\end{proof}

We remark that the statement of the Double Skewering fails if one allows fixed points at infinity. A good example illustrating this is the fixator of a point $\xi $ at infinity in the full automorphism group of a regular tree $T$. This group acts edge-transitively, hence essentially. Every hyperplane is  skewered by a hyperbolic isometry. However, {a pair $\k \subset \h$ of half-spaces such that $\h$ contains $\xi$ but $\k$ does not, is not skewered by a common hyperbolic isometry.} 

{
A final ingredient in the proof of Rank Rigidity for \cat cube complexes is a criterion allowing us to recognize when a \cat cube complex $X$ is \textbf{irreducible}, \emph{i.e.} when it does not split as a product of non-trivial convex subcomplexes. A pair of hyperplanes $\hh_1, \hh_2$ in $X$ is called \textbf{strongly separated} if no hyperplane crosses both  $\hh_1$ and $\hh_2$; in particular  $\hh_1$, $ \hh_2$ must be disjoint. This notion is due to J.~Behrstock and R.~Charney~\cite{BehrstockCharney}.

\begin{IrredCrit}
Let $X$ be a finite-dimensional unbounded \cat cube complex {such that  $\Aut(X)$ acts essentially without a fixed point at infinity. }

Then $X$ is irreducible if and only if there is a pair of strongly separated hyperplanes.
\end{IrredCrit}

A slightly more precise version will be established in Proposition~\ref{prop:IrredCriterion} below. 

With these tools at hand, Theorem~\ref{thmintro:RankRigidity} will be deduced from the simple observation that a hyperbolic isometry which is obtained by applying the Double Skewering Lemma to a strongly separated pair of hyperplanes must necessarily be a contracting isometry (see Lemma~\ref{lem:Contracting} below). 
}

\bigskip
We now proceed to describe various applications of Theorem~\ref{thmintro:RankRigidity}.

\subsection{Lattices and regular elements}

{{When it exists, a contracting}} isometry of a \cat space $X$ can be viewed as an analogue of a \textbf{regular semi-simple element} in the classical case of symmetric spaces, \emph{i.e.} a semi-simple element whose centralizer has minimal possible dimension in the ambient Lie group $\Isom(X)$. Indeed, if $X$ is a proper \cat space and $g \in \Isom(X)$ has rank one, then the centralizer $\centra_{\Isom(X)}(g)$ is of dimension~$\leq 1$ in the sense that it is isomorphic to a closed subgroup of the isometry group of the real line $\RR$.

An important property of regular semi-simple elements in the classical setting is that they form a Zariski open subset of the full isometry group. In particular, combining this with the Borel Density Theorem, it follows that any lattice in a symmetric space contains a regular semi-simple element. The following result is an analogue of that fact. 

\begin{thmintro}[Existence of regular elements]\label{thm:Regular}
Let $X = X_1 \times \dots \times X_n$ be a product of $n$ irreducible unbounded locally compact \cat cube complexes such that $\Aut(X_i)$ acts cocompactly and essentially on $X_i$ for all $i$. 

Then any (possibly non-uniform) lattice $\Gamma \leq \Aut(X)$ contains an element $\gamma \in \Gamma$ which acts as a rank one isometry on each irreducible factor $X_i$. 
\end{thmintro}

As a consequence, one deduces the following result related to the Flat Closing Conjecture. 

\begin{corintro}\label{cor:Zn}
Let $X $ be a locally compact \cat cube complex and $\Gamma$ be a discrete group acting cocompactly on $X$. If $X$ is a product of $n$ {unbounded} cube subcomplexes, then $\Gamma$ contains a subgroup isomorphic to $\ZZ^n$. 
\end{corintro}

\subsection{Euclidean alternative}

\sloppy
A simple-minded application of the Rank Rigidity of \cat cube complexes is the following characterisation of those \cat cube complexes whose full automorphism group stabilises some isometrically embedded Euclidean flat. 
By a \textbf{facing triple} of hyperplanes, we mean a triple of hyperplanes associated to a triple of pairwise disjoint half-spaces.

\begin{thmintro}\label{thmintro:PseudoEucl}
Let $X$ be a finite-dimensional \cat cube complex such that $\Aut(X)$ acts essentially without fixed point at infinity. Then either $\Aut(X)$ stabilises a Euclidean flat (possibly reduced to a single point)  or there is a facing triple of hyperplanes. 
\end{thmintro}

{
In the special case when $X$ is locally compact and $\Aut(X)$ acts cocompactly, the same holds even if $\Aut(X)$ has some fixed points at infinity, see Theorem~\ref{thm:PseudoEucl} below. 
}

\subsection{Tits alternative}

One of the most natural uses of the strong dynamical properties of contracting isometries is to produce Schottky pairs generating (discrete) free subgroups via the Ping Pong Lemma. In particular, there is a relation between the Rank Rigidity Conjecture for proper \cat spaces and the Tits Alternative, which constitutes another major open problem in this area. 

In the case of finite-dimensional \cat cube complexes, a version of the Tits Alternative was already obtained by M.~Sageev and D.~Wise \cite{SageevWise}, not using contracting elements but relying instead on the Algebraic Torus Theorem. 

Here we present two other versions of the Tits Alternative relying on Rank Rigidity. We emphasize that our approach is here purely geometric; {most arguments use only the combinatorics of hyperplanes in \cat cube complexes}.

\begin{thmintro}[Tits Alternative, first version]\label{thm:TitsAlt}
Let $X$ be a finite-dimensional \cat cube complex and let $\Gamma \leq \Aut(X)$. Then $\Gamma$ has a finite index subgroup fixing a point in $X \cup \bd X$ or $\Gamma$ contains a non-Abelian free subgroup. 
\end{thmintro}

\begin{proof}
Assume that $\Gamma$ does not virtually fix any point in  $X \cup \bd X$. It follows that there is a $\Gamma$-invariant {convex} subcomplex $Y \subset X$ on which $\Gamma$ acts essentially and without a fixed point at infinity (see Proposition~\ref{prop:pruning} below). 

Every finite-dimensional \cat cube complex $Y$ has a canonical decomposition $Y = Y_1 \times \dots \times Y_n$ as a finite product of irreducible subcomplexes, which is preserved by the full automorphism group $\Aut(Y)$ up to permutations of possibly isomorphic factors (see Proposition~\ref{prop:deRham} below). After replacing $\Gamma$ by a finite index subgroup, we obtain in particular an essential action of $\Gamma$ on each irreducible factor $Y_i$ which does not fix any point at infinity. 

Assume that for some $i$, the cube complex $Y_i$ does not contain any Euclidean flat which is invariant under $\Aut(Y_i)$. Then Theorem~\ref{thmintro:PseudoEucl} ensures the existence of a facing triple of hyperplanes in $Y_i$. We then invoke the Flipping Lemma, which shows that there actually exists a facing \emph{quadruple} of hyperplanes. We group these four hyperplanes into two pairs and apply the Double Skewering Lemma to each of them. This provides a Schottky pair of hyperbolic isometries, from which the existence of a non-Abelian free subgroup in $\Gamma$ follows via the Ping Pong Lemma. 

Assume now that for all $i$, the factor $Y_i$ contains some $\Aut(Y_i)$-invariant flat $F_i \subset Y_i$. Since $Y_i$ is irreducible, it follows that $F_i$ is one-dimensional (see Lemma~\ref{lem:pseudoEucl} below). Thus $\Aut(Y_i)$ has an index two subgroup which fixes a point at infinity, contradicting our assumption on $\Gamma$. 
\end{proof}

The first alternative in Theorem~\ref{thm:TitsAlt} might look unsatisfactory in the sense that it does not provide any algebraic information on $\Gamma$. Notice however that no assumption on local compactness of $X$ is made. Therefore, the above statement is optimal because \emph{any} group admits an action on a tree $X$ with a global fixed point (take $X$ to be an infinite star on which $\Gamma$ acts by fixing the root and permuting the branches). It is nevertheless possible to combine Theorem~\ref{thm:TitsAlt} with some results from \cite{CapraceLytchak} to obtain a more precise description of $\Gamma$ in the case of a proper action. 

\begin{corintro}[Tits Alternative, second version]\label{cor:TitsAlt}
Let $X$ be a finite-dimensional \cat cube complex and $\Gamma$ be a discrete group acting properly on $X$. Then either $\Gamma$ is \{locally finite\}-by-\{virtually Abelian\} or $\Gamma$ contains a non-Abelian free subgroup.
\end{corintro}

This is a slight improvement of the statement from \cite{SageevWise}. 

\subsection{Dynamics on the boundary}

It is a well known fact that if a group $\Gamma$ acts on a proper \cat space $X$ and contains a rank one isometry, then the $\Gamma$-action on its limit set is \textbf{topologically minimal} (\emph{i.e.} every orbit is dense) and the limit set is the unique $\Gamma$-minimal subset of the boundary $\bd X$ (provided the $\Gamma$-action is \textbf{non-elementary}, \emph{i.e.} there is no global fixed point or fixed pair at infinity). We refer to \cite{BallmannBuyalo} and \cite{Hamenstadt} for more information.

In fact, according to recent results by Gabriele Link~\cite{Link}, even if $X$ is a product space, it is possible to obtain very precise dynamical information on the $\Gamma$-action on its limit set if one knows that $\Gamma$ contains an element which acts as a rank one isometry on each irreducible factor of $X$. In particular, when $X$ is a \cat cube complex, all the results recently obtained in the aforementioned paper by Gabriele Link apply to \cat cube complexes under the hypotheses of Theorem~\ref{thm:Regular}. {We shall not repeat these statements here; the interested reader should consult \cite{Link} and references therein. }

\subsection{Quasi-morphisms and products of trees}

Another way to exploit the peculiar dynamical properties of contracting isometries has recently been elaborated by M.~Bestvina and K.~Fujiwara \cite{BestvinaFujiwara} in order to construct quasi-morphisms. In particular, combining the Bestvina--Fujiwara construction with Theorem~\ref{thmintro:RankRigidity}, it follows that \emph{a discrete group acting properly, essentially and without fixed point at infinity on a locally compact finite-dimensional \cat cube complex is either virtually Abelian or has an infinite-dimensional space of non-trivial quasi-morphisms}. Finer results applying also to non-discrete actions on non-proper spaces can be further obtained; we shall not give more details in this direction here. Instead, we present the following ``\emph{bounded cohomological characterisation}'' of product of trees, which is in the same spirit as the main result from  \cite{BestvinaFujiwara}.

\begin{thmintro}\label{thm:ProductTrees}
Let $X = X_1 \times \dots \times X_n$ be a product of $n$ irreducible locally compact geodesically complete \cat cube complexes  such that $\Aut(X_i)$ acts cocompactly on $X_i$ for all $i$. 
For any (possibly non-uniform) lattice $\Gamma \leq \Aut(X_1)\times \dots \times \Aut(X_n)$, the following conditions are equivalent. 
\begin{enumerate}[(i)]
\item $\QH(\Gamma) $ is finite-dimensional. 

\item $\QH(\Gamma) = 0$. 

\item For all $i$, the space $X_i$ is a semi-regular tree, and if $X_i$ is not isometric to the real line,  then the closure $G_i$ of the projection of $\Gamma$ to $\Aut(X_i)$ is doubly transitive on $\bd X_i$. 
\end{enumerate}
\end{thmintro}

Recall that $\QH(\Gamma)$ denotes the real vector space of quasi-morphisms of $\Gamma$ modulo the subspace of trivial quasi-morphisms, which is the direct sum of the space of bounded functions on $\Gamma$ and the space of genuine morphisms $\Gamma \to \RR$. The space $\QH(\Gamma)$ coincides with the kernel of the canonical map from the bounded cohomology to the usual cohomology of $\Gamma$ in degree two with trivial coefficients. 

We point out that there is a version of Theorem~\ref{thm:ProductTrees} which does not require the assumption of geodesic completeness. In that case the assertion (iii) must be replaced by the fact that $X_i$ is equivariantly quasi-isometric to some tree on which $G_i$ acts naturally by automorphisms (see Remark~\ref{rem:QH} below). 

As we shall see below, the details of the proof of Theorem~\ref{thm:ProductTrees}  have little to do with \cat cube complexes, but rather depend on a combination of Rank Rigidity with general arguments on quasi-morphisms and \cat spaces with rank one isometries (see notably Theorem~\ref{thm:QH} below).

\subsection{Asymptotic cones}

Another typical consequence of the existence of rank one isometries concerns the aqf the corresponding groups and spaces. This phenomenon was studied in detail by C.~Drutu, Sh.~Mozes and M.~Sapir \cite{DMS}. Relying on their work, it is straightforward to deduce the following consequence of Rank Rigidity.

\begin{corintro}\label{cor:AsymCones}
Let $X$ be a locally compact \cat cube complex such that $\Aut(X)$ acts cocompact, essentially and without a fixed point at infinity. Then
\begin{enumerate}[(i)]
\item $X$ is irreducible if and only if every asymptotic cone of $X$ has a cut-point. 

\item If $X$ is irreducible and   $\Gamma \leq \Aut(X)$ is a finitely generated group acting essentially  without a fixed point at infinity, then every asymptotic cone of $\Gamma$ has a cut-point.
\end{enumerate}
\end{corintro}

{
\bigskip 
The paper is organized as follows. Sections~\ref{sec:prel} and~\ref{sec:essential} are of preliminary nature. They collect a number of useful general facts on groups acting on \cat cube complexes, most of which are well-known to the experts. The new material is exposed in the following sections. The proof of the Flipping Lemma is given in Section~\ref{sec:Flipping} and we recommend to readers who have some familiarity with \cat cube complexes to start reading this paper from that point on. In fact Section~\ref{sec:Flipping} contains two distinct (and conceptually different) proofs of the Flipping Lemma, the first one shorter and more conceptual, the second more pedestrian but entirely self-contained. Section~\ref{sec:StronglySeparated} is devoted to a criterion allowing one to recognise irreducible \cat cube complexes. With this  criterion at hand, we complete the proof of the Rank Rigidity theorem in Section~\ref{sec:RkOne} and present the applications in the final section. 

Let us finally point out that part of the technical difficulties in this paper come from the fact that we have tried to let our arguments work in the most general setting possible, rather than focusing on the special case of a locally compact \cat cube complex endowed with a proper cocompact action of some discrete group. The reader who is primarily interested in the latter situation will realise that  many of our discussions can be simplified to a large extent in that specific case; we nevertheless decided to include discussions of finite-dimensional spaces that are possibly non-proper, or of lattices that are possibly non-uniform. {This level of generality is useful even in the study of proper actions. For example, proper actions on products may descend to improper actions on their factors.}
}

\subsection*{Acknowledgement} The first-named author would like to thank the Moshav Shorashim Research Institute for its hospitality while a large part of the present work was accomplished. {We express our gratitude to Frédéric Haglund and the referee for long lists of detailed comments on an earlier version of this manuscript, which helped much in improving the presentation.}

\section{Preliminaries}\label{sec:prel}

We first recall some basic facts about cube complexes and their connection to the hyperplanes and half-plane systems. For more details see \cite{Roller}, \cite{Nica}, \cite{ChatterjiNiblo}, \cite{Guralnik}
 or \cite{Sageev95}.

{
\subsection{\cat cube complexes}
Let $X$ be a \cat cube complex, \emph{i.e.} a simply connected  cell complex all of whose cells are Euclidean cubes with edge length one, and such that the link of each vertex is a flag complex. A theorem of M.~Gromov (see \emph{e.g.} Theorem~II.5.20 from \cite{Bridson-Haefliger} for the finite-dimensional case and Theorem~40 from \cite{Leary} for the general case) ensures that $X$, endowed with the induced length metric, is a \cat space. Moreover  $X$ is complete if and only if $X$ does not contain an ascending chain of cells (see \cite[Theorem~31]{Leary}); in particular, this is the case if $X$ is \textbf{locally finite-dimensional}, in the sense that the supremum of the dimensions of cubes containing any given vertex is finite. We shall later focus on \cat cube complexes which are \textbf{finite-dimensional}, which means that the supremum of the dimensions of all cubes in $X$ is finite. The space $X$ is locally compact if and only if it is locally finite, in the sense that every vertex has finitely many neighbours. 
 
Unless specified otherwise, we shall view $X$ as a metric space with respect to its natural \cat metric (rather than focusing on the $1$-skeleton of $X$ endowed with the combinatorial metric). In some cases, the convexity properties of the \cat metric constitute an advantage compared to the combinatorial metric (this is for example the case in the proof of Theorem~\ref{thm:Flipping} below), while in other cases, the combinatorial metric is easier to deal with by its very combinatorial nature (as in {Lemma~\ref{lem:KeyContracting}} below). However, both viewpoints are very often equivalent; we refer to Lemma~\ref{lem:cat_versus_combinatorial} below for a more precise statement.

A subcomplex $Y$ of $X$ is called \textbf{convex} if it is convex as a \cat subspace. In that case $Y$ is itself a \cat cube complex. 

If $Y$ and $Z$ are two \cat cube complexes, then the Cartesian product $X = Y \times Z$ is naturally endowed with the structure of a \cat cube complex: the $1$-skeleton of $X$ is the graph theoretical Cartesian product of the $1$-skeletons of $Y$ and $Z$, and the metric on $X$ coincides with $\sqrt{d_Y^2 + d_Z^2}$, where $d_Y$ and $d_Z$ are the respective metrics on $Y$ and $Z$. For all vertices $y \in Y $ and $z \in Z$, the subcomplexes $Y \times \{z\}$ and $\{y\} \times Z$ are convex in $X$. 
 }

{By a \textbf{cubical map} between cube complexes, we mean a cellular map so that the restriction $\sigma\to\tau$ between cubes factors as $\sigma\to\nu\to\tau$, where the first map $\sigma\to\nu$ is a natural projection onto a face of $\sigma$ and the second map $\nu\to\tau$ is an isometry.}

\subsection{Hyperplanes} 

{A key feature of \cat cube complexes is the existence of \textbf{hyperplanes}. In fact, the midpoint of each edge in the $1$-skeleton of $X$ belongs to a unique hyperplane. A hyperplane $\hh$ is a closed convex subspace which has the property that its complement $X  \setminus \hh$ has exactly two connected components, both of which are convex. The closure of each of these two components is  called a \textbf{half-space}.} 
The set of hyperplanes of $X$ is denoted
by $\hH(X)$ and the set of half-spaces by $\H(X)$. The hyperplane associated to a half-space $\h$ is denoted by $\hh$ and the half-space complementary to $\h$ by $\h^*$. {By convention, we shall always use gothic characters to denote half-spaces. 
}

We will mostly make the assumption that $X$ is finite-dimensional. One way this assumption can be used is through the fact that any collection of pairwise crossing hyperplanes has a non-empty intersection (see \cite[Theorem~4.14]{Sageev95}). Therefore, the cardinality of this set is bounded above by $\dim(X)$. {This fact shall often be used through  the following. 

\begin{lem}\label{lem:dim}
Let $X$ be a finite-dimensional \cat cube complex. For each $k>0$, there is some $R(k)>0$ such that if a geodesic path $\alpha$ in $X$ crosses at least $R(k)$ hyperplanes, then $\alpha$ crosses  a  \textbf{pencil of  hyperplanes}  of cardinality at least~$k$, \emph{i.e.} a collection of at least $k$ disjoint hyperplanes which are associated to a collection of $k$ nested half-spaces.
\end{lem}

\begin{proof}
In view of Ramsey's theorem, it suffices to define $R(k)$ as the Ramsey number $R(k, \dim(X)+1)$. 
\end{proof}

This implies the following basic observation. 
\begin{lem}\label{lem:cat_versus_combinatorial}
Let $X$ be a finite-dimensional \cat cube complex. Then $X$ is quasi-isometric to its $1$-skeleton endowed with the combinatorial metric. 
\end{lem}
\begin{proof}
Let $d$ denote the \cat metric on $X$ and $d_{\ell^1}$ denote the combinatorial metric on the vertex set $X^{(0)}$. 
Let also $n = \dim(X)$. Then every points of $X$ is at distance at most~$\sqrt n/2$ from a vertex of $X$. Thus it suffices to show that $(X^{(0)}, d)$ is quasi-isometric to $(X^{(0)},d_{\ell^1})$. 

The definition of $d$ implies that $d(x, y) \leq d_{\ell^1}(x, y)$ for all $x, y \in X^{(0)}$. Moreover $d_{\ell^1}(x, y)$ coincides with the number of hyperplanes separating $x$ from $y$. By Lemma~\ref{lem:dim}, there is a constant $R = R(2)$ such that any collection of $R$ hyperplanes contains a pair of non-crossing hyperplanes. The minimal distance separating a pair of non-crossing hyperplanes in a \cat cube complex is~$1$. This implies that for all vertices $x, y \in X^{(0)}$, we have $ d(x, y) \geq  d_{\ell^1}(x, y)/ R$. Thus $d(x, y) \leq d_{\ell^1}(x, y) \leq Rd(x, y)$, as desired. 
\end{proof}
}

\subsection{Pocsets}\label{Pocsets}

{As before, we assume for this discussion that $X$ is a finite-dimensional \cat cube complex, although some of what is discussed below can be generalized to a somewhat broader setting.} The set of half-spaces is a poset under inclusion 
and comes equipped with a natural order reversing involution {$\h\to\h^*$} (hence is called a \textbf{pocset}, see \cite{Roller}). This pocset also satisfies the \textbf{finite interval condition}, meaning that if $\h_1\subset\h_2$ are half-spaces, there are only finitely many half-spaces $\h$ satisfying $\h_1\subset \h\subset\h_2$. {The finite dimensionality of $X$ is reflected in the fact that the pocset has \textbf{finite width}, meaning that there is a bound on the size of a collection of \textbf{transverse} elements: halfspaces associated to intersecting hyperplanes.  }

Now there is also a dual construction which shows that $X$ can be completely reconstructed from the pocset $\H(X)$. Given a pocset $\Sigma$ satisfying the finite interval condition {and having finite width}, one can construct a CAT(0) cube complex $X=X(\Sigma)$ whose half-space system is naturally isomorphic to $\Sigma$.  {It is then an easy exercise to check that} the vertices of a $X$ are in $1$-$1$ correspondence with ultrafilters on $\Sigma$ which satisfy the descending chain condition. Namely a vertex $v$ in $X^0$ may be viewed as a subset of $\Sigma$ which satisfies 
\begin{enumerate}
\item For every involutary invariant pair of elements $\{A,A^*\}$, exactly one of them is in $v$.
\item If $A< B$ and $A\in v$ then $B\in v$.
\item Every descending chain of elements of $v$ terminates. 
\end{enumerate}

One joins two such vertices by an edge when they differ on a single involutary pair. That is, $v$ and $w$ are joined by an edge if there exists a $\h$ such that $w=(v-{\h})\cup\h^*$. One then attaches a cube whenever the $1$-skeleton of one appears in $X^{1}$. It is then a theorem that the resulting cube complex $X$ is \cat.  {(By convention, if the pocset if empty then $X$ is reduced to a single vertex.)} In fact, a theorem of Roller~\cite{Roller} 
then tells us that \cat cube complexes and pocsets are actually dual to one another. If one starts with a {finite dimensional} \cat cube complex $X$, then $X(\H(X))=X$. That is the cube complex constructed from the pocset of half-spaces associated to the hyperplanes of $X$ is $X$. Conversely, for any pocset $\Sigma$, we have $\H(X(\Sigma))=\Sigma$. That is, if one starts with a pocset and builds a cube complex, the pocset of half-spaces associated to the cube complex is again the original pocset. 

Now let  $X$ be a \cat cube complex, let $\hH=\hH(X)$ be the collection of hyperplanes of $X$, and let  
$\hat\K\subset\hH$ be some subcollection of hyperplanes. One can carry out the above construction to build a new cube complex $X(\hat\K)$ whose half-space system is isomorphic to $\K$, the {collection of half-spaces bounded by some $\hh  \in \hat\K$}. One then obtains a natural quotient map $X\to X(\hat\K)$. Namely a vertex in $X$ corresponds to a choice of half-spaces in $\H$ and thus gives rise to a choice of half-spaces on $\hat \K$, and it is immediate  to see that the resulting ultrafilter $\hat \K$ is a vertex of $X(\hat\K)$. This map on vertices is easily seen to extend to a  {cubical map $X\to X(\hat\K)$}. We call this map the \textbf{restriction quotient} 
arising from the subset $\hat\K\subset\hH$.  If $X$ furthermore comes equipped with a group action $\Gamma$ and $\hat\K$ is a $\Gamma$-invariant subset of $\hH$, then the restriction quotient is naturally equipped with a $\Gamma$-action and the quotient map is equivariant. For example, if $\hat \K$ is the set of all hyperplanes that cross some given hyperplane $\hat\h$, then $X(\hat \K)$ is isometric to the hyperplane $\hat\h$ endowed with its canonical structure of a cube complex inherited from $X$. 

It is important to remark that the dimension of the cube complex $X(\hat\K)$ is bounded above by $\dim(X)$, but that  $X(\hat\K)$ need not be proper even if $X$ is so. 

We now consider two other applications of this construction.

\subsection{Application: Orbit quotients and skewers} 

Let $\Gamma \leq \Aut(X)$ be any group of automorphisms of a \cat cube complex $X$. We consider the orbit $\Gamma\hh$ of a single hyperplane $\hh\in\hH$. We may then form the restriction quotient $X(\Gamma\hh)$, which we call the \textbf{orbit quotient} of $\hh$.  

In order to present a useful application of this construction, we first introduce an important definition. 
Given a hyperplane $\hh \in \hH(X)$ and an isometry $\gamma \in \Aut(X)$, we say that $\gamma$ \textbf{skewers} $\hh$ (or that $\hh$ is \textbf{skewered} by $\gamma$) if  there exist a non-zero integer $n$ {and a half-space $\h$ bounded by $\hh$ such that $\gamma^n \h \subsetneq \h$ (the proper inclusion is essential).}

Notice that an element $\gamma$ which skewers some hyperplane necessarily has positive translation length. In particular, if $X$ is finite-dimensional, then $\gamma$ acts as a hyperbolic isometry {(in the usual sense of \cat geometry)}. In that case, we have the following criterion. 

\begin{lem}\label{lem:def:Skewer}
Let $X$ be a finite-dimensional \cat cube complex, let $\gamma \in \Aut(X)$ be a hyperbolic isometry and $\hh \in \hH(X)$. The following conditions are equivalent.
\begin{enumerate}[(i)]
\item $\gamma$ skewers $\hh$. 

\item Every $\gamma$-axis meets the hyperplane $\hh$ in a single point.

\item Some $\gamma$-axis meets the hyperplane $\hh$ in a single point.
\end{enumerate}
\end{lem}

\begin{proof}
First, note that an axis for $\gamma$ and a hyperplane can intersect in at most one point, unless the axis is contained in the hyperplane. 

The equivalence of (ii) and (iii) is clear as any two axes for a given hyperbolic isometry are a bounded distance apart. To see (i)$\Rightarrow$(iii) suppose that $\h$ is such that 
$\gamma^n\h\subsetneq \h$. Let $\ell$ be an axis for $\gamma$ and let $p$ be a point on $\ell$ which is not in $\hh$. Without loss of generality let us assume that $p\in\h$. Let $A$ be some geodesic arc joining $p$ and $\hh$. Note that $A$ crosses only finitely many hyperplanes so that for some $m>0$, we have that $\gamma^{nm}\hh\cap A=\varnothing$. It follows that $A\subset\gamma^{mn}(h^*)$. Thus $p\in\gamma^{nm} \h^*$ which implies that $\gamma^{-nm}(p)\in \h^*$. Since $p$ was chosen on $\ell$ and $\ell$ is $\gamma$-invariant, it follows that $\ell$ meets both $\h$ and $\h^*$, as required. 

Conversely, assume that some axis $\ell$ of $\gamma$ meets $\hh$ in a single point $p$. Consider the collection of positive powers $\{\gamma^n\hh\vert n>0\}$. This is an infinite collection of hyperplanes since the collection $\{\gamma^n p\vert n>0\}$ is infinite. By finite dimensionality, there exists some $n,m>0$ such that $\gamma^n\hh\cap\gamma^m\hh=\varnothing$. Thus, there exists some positive power $l$ such that $\gamma^l\hh\cap\hh=\varnothing$. If $\h$ is the halfspace bounded by $\hh$ containing $\gamma^l(p)$, then we have that $\gamma^l\h\subset \h$, as required. 
\end{proof}

 We move on with the following useful description of orbit quotients.

\begin{lem}\label{lem:WallOrbit}
Let $X$ be a finite-dimensional \cat cube complex, let $\Gamma \leq
\Aut(X)$ and $\hh \in \hH(X)$. We have the following.
\begin{enumerate}[(i)]
\item $X(\Gamma\hh)$ is finite if and only if the orbit $\Gamma\hh$
is finite.

\item $X(\Gamma\hh)$ is bounded if and only if $\Gamma$ has a fixed
point in $X(\Gamma\hh)$.

\item $X(\Gamma\hh)$ is unbounded if and only if $\hh$ is skewered by some hyperbolic isometry $\gamma \in \Gamma$. 
\end{enumerate}
\end{lem}

\begin{proof}
The point (i) is clear, and so is the `only if' of (ii). Suppose
that $\Gamma$ fixes a point $v$ in $X(\Gamma\hh)$. If $v$ belong to
some  hyperplane of $X(\Gamma\hh)$, then every  hyperplane of $X(\Gamma\hh)$
contains $v$. This implies that $X(\Gamma\hh)$ consists of a single
cube, which is bounded since $\dim(X(\Gamma\hh)) \leq \dim(X) < \infty$.
Otherwise we may assume that $v$ is a vertex of $X(\Gamma\hh)$. Then
every half-space of $X(\Gamma\hh)$ separates $v$ from a neighbouring
vertex. This implies that any minimal path in the $1$-skeleton of
$X(\Gamma\hh)$ emanating from $v$ remains in a single cube containing
$v$. Since the dimension of such a cube is bounded, we deduce that
the $1$-skeleton of $X(\Gamma\hh)$ is contained in a finite
ball centred at $v$. Hence $X(\Gamma\hh)$ is bounded as desired.

It remains to prove (iii). If $\gamma$ skewers $\hh$, then {$\la \gamma \ra.\h$ contains an infinite chain of half-spaces by Lemma~\ref{lem:dim} and, hence the cube complex  $X(\la \gamma \ra.\hh)$ is unbounded.} Consequently, so is $X(\Gamma \hh)$. Assume conversely that $X(\Gamma \hh)$ is unbounded. Let $d$ denote its dimension. By Lemma~\ref{lem:dim},  a geodesic edge path in $X(\Gamma \hh)$ whose length is greater than the constant   $R(3)$ necessarily crosses three pairwise disjoint hyperplanes  $\hh_1$, $\hh_2$ and $\hh_3$. Since there is a single orbit of hyperplanes in $X(\Gamma\hh)$, we infer that 
{two of the half-spaces bounding the $\hh_i$'s are nested,} 
giving rise to a hyperbolic element in $\Gamma$ skewering $\hh$, as required. 
\end{proof}

\subsection{Application: Products} 
Suppose that $X$ is a \cat cube complex which factors as a product of two \cat cube complexes $X=X_1\times X_2$. Then we have two natural projection maps $p_i:X\to X_i$ and the hyperplanes of $X$ is partitioned as a disjoint union $\hH=\hH_1\cup\hH_2$ where $\hH_i$ is composed of the set of hyperplanes of $X_i$ pulled by back to $X$ by the projection map $p_i$. {Notice that $\hH_i$ is empty if and only if $X_i$ is reduced to a single vertex.} We now observe that $\hH_1$ and $\hH_2$ are transverse, meaning that every hyperplane in $\hH_1$ crosses every hyperplane in $\hH_2$. Thus, the pocset $\H=\H(X)$ admits an involution invariant decomposition $\H=\H_1\cup\H_2$ such that the each half-space in $\H_1$ is incomparable to any half-space in $\H_2$.  The projection map $p_i : X \to X_i$ is nothing but the restriction quotient arising from the subset $\hH_i \subset \hH$. In other words we have $X_i \cong X(\hH_i)$.

Conversely, we could start with two pocsets $\Sigma_1$ and $\Sigma_2$  satisfying the finite interval condition and take their disjoint union to obtain a new such pocset 
$\Sigma=\Sigma_1\cup\Sigma_2$, where each element of $\Sigma_1$ is incomparable to an each element in $\Sigma_2$. Let $X$ denote the cube complex associated to $\Sigma$ and $X_i$ denote the cube complex associated to $\Sigma_i$. Then we have the pocset of half-spaces associated to $X$ is isomorphic to the pocset of half-spaces associated to $X_1\times X_2$. Since the construction of the \cat cube complex from the pocset is canonical, it follows that $X=X_1\times X_2$. To summarize, we have the following.

\begin{lem}\label{lem:Prod}
A decomposition of a \cat cube complex as a product of cube complexes corresponds to a partition of the collection of hyperplanes of $X$,  $\hH=\hH_1\cup\hH_2$  such that every hyperplane in $\hH_1$ meets every hyperplane in $\hH_2$. 
\end{lem}

A cube complex which cannot be decomposed as above is called \textbf{irreducible}. 
An easy consequence is the following basic analogue of the de Rham decomposition theorem. 

\begin{prop}\label{prop:deRham}
A finite-dimensional \cat cube complex $X$ admits a canonical decomposition 
$$X = X_1 \times \cdots \times X_p$$
into a product of irreducible cube complexes $X_i$. Every automorphism of $X$ preserves that decomposition, up to a permutation of possibly isomorphic factors.  In particular, the image of the canonical embedding 
$$\Aut(X_1) \times \cdots \times \Aut(X_p) \hookrightarrow \Aut(X)$$
has finite index in $\Aut(X)$.
\end{prop}
\begin{proof}
Since $X$ is finite-dimensional, any product decomposition can be refined into a \emph{finite} product of \emph{irreducible} factors. Therefore it suffices to show that if $X $ admits two such decompositions $X = X_1 \times \cdots \times X_p$ and $X = X'_1 \times \cdots \times X'_q$ then $p = q$ and $X_i = X'_{\sigma(i)}$ for some permutation $\sigma$ of $\{1, \dots, p\}$. Consider the partitions of $\hH$ corresponding respectively to these product decompositions: $\hH = \hH_1 \cup \dots \cup \hH_p$ and $\hH = \hH'_1 \cup \dots \cup \hH'_q$, where $\hH_i  = \hH(X_i)$ and $\hH'_j = \hH(X'_j)$. The second partition of $\hH$ induces a partition of each individual subset $\hH_i$. Lemma~\ref{lem:Prod} ensures that this must be the trivial partition since $X_i$ is irreducible. In particular we infer that $p \leq q$. By symmetry we have $p \geq q$ and the desired result follows easily.
\end{proof}

\section{Essential things}\label{sec:essential}

\subsection{Essential hyperplanes}

Let $X$ be a \cat cube complex. A half-space $\h$ is called \textbf{deep} if it properly contains arbitrarily large balls of $X$. Otherwise we say it is \textbf{shallow}. This allows us to break up the set $\hH$ into three types of hyperplanes: \textbf{essential}, \textbf{half-essential} and \textbf{trivial}, as follows. A hyperplane $\hh$ is called \textbf{essential} if the half-spaces $\h$ and $\h^*$ are both deep. If $\h$ and $\h^*$ are both shallow, then $\hh$ is called \textbf{trivial}. If $\hh$ is neither essential nor trivial, we call it \textbf{half-essential}. We let $\Ess(X)$ (resp. $\Hess(X)$, $\Triv(X)$) denote the collection of essential (resp. half-essential, trivial) hyperplanes.

\subsection{Essential cube complexes}

A \cat cube complex $X$ is called \textbf{essential} if all its hyperplanes are essential or, in other words, if $\hH(X) = \Ess(X)$. We define the \textbf{core} of $X$ as the restriction quotient $X(\Triv(X) \cup \Ess(X))$ and the \textbf{essential core} of $X$  as the restriction quotient $X(\Ess(X))$. Clearly the essential core is always essential; it is endowed with a canonical $\Aut(X)$-action. Notice however that the core of $X$ might be reduced to a single point in general, even if $X$ is unbounded.  For example, consider the standard squaring of the Euclidean quarter-plane, in which all hyperplanes are half-essential.  In order to deal with that issue, we shall analyze the notions introduced thus far relatively to the action of a subgroup $\Gamma \leq \Aut(X)$.

\subsection{Essential hyperplanes relative to a group action}

Let $\Gamma \leq \Aut(X)$. 
Choose a vertex $v\in X$. A half-space $\h$ is called \textbf{$\Gamma$-deep} if it contains orbit points of $v$ arbitrarily far from $\hh$ (note that this definition is independent of the choice of $v$).  Otherwise $\hh$ is called \textbf{$\Gamma$-shallow}. Mimicking the above definitions, we again break up the set $\hH$ into three types of hyperplanes:   \textbf{$\Gamma$-essential}, \textbf{$\Gamma$-half-essential} and \textbf{$\Gamma$-trivial}.  We also define the symbols $\Ess(X,\Gamma)$,  $\Hess(X,\Gamma)$ and $\Triv(X,\Gamma)$ accordingly, as well as the \textbf{$\Gamma$-core} of $X$ and the \textbf{$\Gamma$-essential core}. 

{The following statement, which we shall use repeatedly and often implicitly in the sequel, clarifies how these notions behave with respect to invariant convex subcomplexes. The proof is straightforward and will be omitted. 

\begin{lem}\label{lem:Haglund}
Let $X$ be a \cat cube complex, let $\Gamma \leq \Aut(X)$ and  $Y \subseteq X$ be a $\Gamma$-invariant convex subcomplex. 
\begin{enumerate}[(i)]
\item  Each hyperplane of $Y$ extends to a unique hyperplane of $X$, so that there is a natural inclusion $\hH(Y) \subseteq \hH(X)$. 

\item Each half-space $\h$ disjoint from $Y$ is $\Gamma$-shallow. In particular we have $\Ess(Y,\Gamma)= \Ess(X,\Gamma)$, and the $\Gamma$-essential core of $Y$ identifies with the $\Gamma$-essential core of $X$. 
\end{enumerate}
\end{lem}
}

The property of being $\Gamma$-essential, half-essential or trivial,  can be recognized in the orbit quotient, as follows.

\begin{prop}\label{prop:essential}
Let $X$ be a finite-dimensional  \cat cube complex, let $\Gamma \leq \Aut(X)$ and  $\hh \in \hH(X)$. Then the following conditions are equivalent:
\begin{enumerate}[(i)]
\item $\hh \in \Ess(X, \Gamma)$,

\item $\hh$ is skewered by some element of $\Gamma$,

\item $X(\Gamma\hh)$ is unbounded.
\end{enumerate}
If in addition $X$ is locally compact, then:
\begin{enumerate}[(i)]
\setcounter{enumi}{3}
\item $\hh\in \Triv(X, \Gamma)$ if and only if $X(\Gamma\hh)$ is finite.

\item $\hh \in \Hess(X, \Gamma)$  if and only if $X(\Gamma\hh)$ is infinite and 
bounded.
\end{enumerate}
\end{prop}

We shall make use of the following basic fact. 

\begin{lem}\label{lem:HalfEssential}
Let $X$ be a finite-dimensional  \cat cube complex, let $\Gamma \leq \Aut(X)$. Let $\hh$ be a hyperplane such that $X(\Gamma.\hh)$ is bounded. 

Then one of the following assertions holds true.
\begin{enumerate}[(i)]
\item $\bigcap_{\gamma \in \Gamma} \gamma.\hh$ is non-empty and $\hh$ is $\Gamma$-trivial.

\item $\bigcap_{\gamma \in \Gamma} \gamma.\h$ contains some vertex of $X$ and $\h^*$ is $\Gamma$-shallow.

\item $\bigcap_{\gamma \in \Gamma} \gamma.\h^*$ contains some vertex of $X$ and $\h$ is $\Gamma$-shallow.
\end{enumerate}
\end{lem}

\begin{proof}
Since $X(\Gamma\hh)$ is bounded, the group $\Gamma$ has a
fixed point $v$ in $X(\Gamma\hh)$. If this belongs to $\hh$, then it is
contained in every element of $\Gamma\hh$. Therefore $X(\Gamma\hh)$ consists of a single cube, {and the orbit $\Gamma\hh$ must therefore be finite. 
Since the hyperplanes in the orbit $\Gamma\hh$  cross pairwise in $X(\Gamma\hh)$, they also cross pairwise in $X$. Therefore, Theorem~4.14 from \cite{Sageev95} ensures that  $\bigcap_{\gamma \in \Gamma} \gamma.\hh$ is a non-empty $\Gamma$-invariant convex subset.} In particular some $\Gamma$-orbit lies entirely in the hyperplane $\hh$, which implies that $\hh$ is $\Gamma$-trivial. We are thus in case (i).  

Assume now that $v$ lies strictly in one side of $\hh$, say in the
half-space $\h$ determined by $\hh$. {In that case, the cube supporting $v$ must be pointwise fixed by $\Gamma$, and we may therefore assume that $v$ is a vertex. } Lifting $v$ up to $X$ via the quotient map $X \to X(\Gamma.\hh)$ (see the construction of the orbit quotient in \S\ref{sec:prel}), we infer that the intersection $Z = \bigcap_{\gamma \in \Gamma} \gamma.\h$ is non-empty
in $X$. Thus $Z$ is a $\Gamma$-invariant convex subset of $X$. In particular, some $\Gamma$-orbit is entirely contained in $\h$. Therefore {Lemma~\ref{lem:Haglund} ensures that} $\h^*$ is $\Gamma$-shallow and we are in case (ii). 

Finally if $v$ lies in $\h^*$, then the same argument shows that the situation (iii) occurs.
\end{proof}

\begin{proof}[Proof of Proposition~\ref{prop:essential}]
If $X(\Gamma\hh)$ is unbounded, then some hyperbolic
element of $\Gamma$ skewers $\hh$ by
Lemma~\ref{lem:WallOrbit} and hence $\hh$ is $\Gamma$-essential. 
Conversely, if $X(\Gamma\hh)$ is bounded, then $\h$ or $\h^*$ is $\Gamma$-shallow by Lemma~\ref{lem:HalfEssential} and, thus,  $\hh$ is $\Gamma$-inessential. 

In view of Lemma~\ref{lem:WallOrbit}, this proves the equivalence between (i), (ii) and (iii). Assume now in addition that $X$ is locally compact. 

\medskip 
By Lemma~\ref{lem:WallOrbit}, the cube complex $X(\Gamma\hh)$
is finite if and only if the orbit $\Gamma\hh$ is so. In that case $\Gamma$ has
a finite index which stabilises $\hh$. Thus $\hh$ is $\Gamma$-trivial. 

Conversely if $\hh$ is $\Gamma$-trivial, so is every element in $\Gamma\hh$ and there is some
$r>0$ such that $X \subset \mathcal N_r(\gamma.\hh)$ for each $\gamma \in \Gamma$. In
particular each wall of $\Gamma\hh$ meets the $r$-ball around every
vertex of $X$. Thus $\Gamma\hh$ is finite since $X$ is locally compact. This shows Assertion (iv). The remaining assertion is an immediate consequence of the others.
\end{proof}

It is easy to deduce from the definition that any $\Gamma$-essential hyperplane intersects every $\Gamma$-trivial hyperplane. In view of Lemma~\ref{lem:Prod}, this yields the following. 

\begin{remark}\label{rem:HessEmpty}
Assume that $\Hess(X, \Gamma)$ is empty. Then $X$ splits as a product $Z \times C$ where $\hH(Z) = \Ess(X, \Gamma)$ and $\hH(C) = \Triv(X, \Gamma)$. {It is good to keep in mind that if $\Ess(X, \Gamma) = \varnothing$, then $Z$ is reduced to a single vertex.}
\end{remark}

\subsection{Essential actions and pruning}

The $\Gamma$-action on $X$ is called \textbf{essential} if every hyperplane is $\Gamma$-essential or, equivalently, if $\hH(X) = \Ess(X) = \Ess(X, \Gamma)$. In particular, if $X$ admits some essential group action, then it is essential. {Lemma~\ref{lem:Haglund} shows that if $X$ contains some $\Gamma$-invariant convex subcomplex $Y \subsetneq X$, then the $\Gamma$-action on $X$ cannot be essential. It is natural to address the question whether $X$ contains some non-empty $\Gamma$-invariant convex subcomplex on which the $\Gamma$-action is essential. }
The $\Gamma$-action on the $\Gamma$-essential core of $X$ is always essential, but this core might be reduced to a singleton. {In case it is not, it is not clear \emph{a priori} that it embeds as a convex subcomplex of $X$. The following result ensures that these possible pathologies do not occur under some natural conditions on the $\Gamma$-action. }

\begin{prop}\label{prop:pruning}
Let $X$ be a finite-dimensional  CAT(0) cube complex and let $\Gamma \leq \Aut(X)$. Assume that at least one of the following two conditions is satisfied:
\begin{enumerate}[(a)]
\item $\Gamma$ has  finitely many orbits of hyperplanes.

\item $\Gamma$ has no fixed point at infinity. 
\end{enumerate}
Then the $\Gamma$-essential core of $X$ is unbounded if and only if $\Gamma$ has no fixed point. In that case  the $\Gamma$-essential core embeds as  a {$\Gamma$-invariant} convex subcomplex $Y$ of $X$. 
\end{prop}

A key criterion allowing us to detect the existence of global fixed points at infinity is provided by the following result, which holds in arbitrary finite-dimensional \cat spaces. {One difficulty for finite-dimensional \cat cube complexes which are not proper is that the visual boundary $\bd X$ has no reason to be non-empty \emph{a priori}. In particular an unbounded sequence from $X$ does not accumulate in $\bd X$ in general. This difficulty will be handled by referring to the main result of \cite{CapraceLytchak}; the latter paper is however not necessary in the case of locally compact \cat cube complexes. }

\begin{prop}\label{prop:CL}
Let $X$ be a finite-dimensional complete \cat space (or more generally, a complete \cat space of finite telescopic dimension, see \cite{CapraceLytchak}) and let $\Gamma \leq \Isom(X)$. Let also $\{Y_\alpha\}_{\alpha \in A}$ be a $\Gamma$-invariant collection of closed convex subsets of $X$. 

If for any finite subset $B \subseteq A$, the intersection $\bigcap_{\alpha \in B} Y_\alpha$ is non-empty, then either  $\bigcap_{\alpha \in A} Y_\alpha$ is a non-empty $\Gamma$-invariant subspace or $\bigcap_{\alpha \in A} \bd Y_\alpha \subset \bd X$ is non-empty and contains a canonical circumcentre which is fixed by $\Gamma$. 
\end{prop}

\begin{proof}
Assume that the intersection $\bigcap_{\alpha \in B} Y_\alpha$ is empty. Then Theorem~1.1 from \cite{CapraceLytchak} ensures that $\bigcap_{\alpha \in A} \bd Y_\alpha \subset \bd X$ is a non-empty $\Gamma$-invariant convex subset of $\bd X$ of radius~$\leq \pi/2$. 
{Notice that if $X$ is proper, then this can be established by an elementary and direct argument without referring to \cite{CapraceLytchak}, but simply using the fact that the cone topology makes $\bd X$ into a compact space (see  
\cite[Proposition~3.2]{CapraceMonod1}).} 
The existence of a fixed point in $\bigcap_{\alpha \in A} \bd Y_\alpha \subset \bd X$  is thus a consequence of \cite[Proposition~1.4]{BalserLytchak}, {which can be applied since $\bd X$ is finite-dimensional (see \cite[Proposition~2.1]{CapraceLytchak}).}
\end{proof}

With this criterion at hand, we are able to provide the following. 

\begin{proof}[Proof of Proposition~\ref{prop:pruning}]
Let us consider the half-essential hyperplanes. We want to describe a pruning process by which we retract the complex down to a complex with no half-essential hyperplanes (in analogy with the standard pruning of a tree, which consists in removing valence-1 vertices). 

Assume thus that $\Hess(X, \Gamma)$ is non-empty. Then there is some $\Gamma$-deep half-space $\h$ such that  $\hh$ is $\Gamma$-half-essential. By Proposition~\ref{prop:essential}  and Lemma~\ref{lem:HalfEssential}, the set $\bigcap_{\gamma \in \Gamma} \gamma.\h$ is a non-empty $\Gamma$-invariant convex subset of $X$ which contains some vertex. Therefore, the collection of all vertices contained in  $\bigcap_{\gamma \in \Gamma} \gamma.\h$ spans a $\Gamma$-invariant convex subcomplex, which we shall denote by $Y_1$. 

Since $Y_1$ is a $\Gamma$-invariant convex subcomplex, it follows that there is a canonical embedding 
$$\Hess(Y_1, \Gamma) \subsetneq \Hess(X, \Gamma).$$

\medskip
We now {repeat the above arguments, this time applied to} the cube complex $Y_1$. Hence if there is some $\Gamma$-deep half-space $\h_1$ such that $\hh_1\in \Hess(Y_1,\Gamma)$,  then there is a non-empty $\Gamma$-invariant convex subcomplex $Y_2 \subsetneq Y_1$, and we have $\Hess(Y_2, \Gamma) \subsetneq \Hess(Y_1, \Gamma)$. We can then proceed inductively: provided the set $\Hess(Y_n, \Gamma)$ is non-empty, we find a  $\Gamma$-invariant convex subcomplex $Y_{n+1} \subsetneq Y_n$. There are two cases to consider. 

If the process terminates after finitely many steps, say at step $n$, it means that $\Hess(Y_n, \Gamma)$ is empty. {By Lemma~\ref{lem:Haglund},}
we deduce that $\Ess(Y_n, \Gamma) = \Ess(X, \Gamma)$. By Remark~\ref{rem:HessEmpty}, the subcomplex $Y_n$ splits as a product of a bounded subcomplex and a subcomplex which is isomorphic to that $\Gamma$-essential core $Z$ of $X$. {Either $Z$ is reduced to a single vertex, in which case $\Gamma$ has a fixed point in $Y_n$, or $Z$ contains a hyperplane, hence an essential one, in which case all $\Gamma$-orbits are unbounded.} Thus we are done in this case. 

Assume now that the process never terminates. Since at each step, we remove at least one $\Gamma$-orbit of half-essential hyperplanes, we deduce that the hypothesis (a) cannot be satisfied. Thus (b) holds and $\Gamma$ has no fixed point at infinity. In particular, it follows from Proposition~\ref{prop:CL} that the intersection $Y_{\infty}=\bigcap_{n>0} Y_n$ is non-empty. As an intersection of convex subcomplexes, $Y_{\infty}$ is itself a $\Gamma$-invariant convex subcomplex. At this point, a transfinite induction argument finishes the proof. 
\end{proof}

\subsection{Finitely many orbits of hyperplanes}

A natural question is how to compare the absolute qualifications of  a hyperplane to the corresponding qualifications {relative} to the $\Gamma$-action. For example, it is clear that if the $\Gamma$-action on $X$ is cocompact, then $\Ess(X,\Gamma) = \Ess(X)$,  $\Hess(X,\Gamma) = \Hess(X)$ and $\Triv(X,\Gamma) = \Triv(X)$. This fact can be generalised as follows; for another related statement, see Corollary~\ref{cor:HereditarilyEss} below. 

\begin{prop}\label{prop:essential:bis}
Let $X$ be a finite-dimensional  locally compact \cat cube complex and $\Gamma \leq \Aut(X)$ be a
group acting with finitely many orbits of hyperplanes. 

Then we have $\Triv(X) = \Triv(X, \Gamma)$, $\Hess(X) = \Hess(X, \Gamma)$ and $\Ess(X) = \Ess(X, \Gamma)$, and the set $\Triv(X)$ is finite. 
\end{prop}
\begin{proof}
First, we remark that the inclusion $\Triv(X) \subseteq \Triv(X, \Gamma)$ is obvious. Conversely, let $\hh \in \Triv(X, \Gamma)$. Thus $X(\Gamma \hh)$ is finite by Proposition~\ref{prop:essential} and hence $\Gamma$ has
a finite index subgroup which stabilises $\hh$. Since $\Gamma$ has finitely many orbits on $\hH(X)$, so
does any finite index subgroup. {Upon replacing $\Gamma$ by a finite index subgroup, we may thus assume that $\hh$ is $\Gamma$-fixed. The distance to $\hh$ is thus a $\Gamma$-invariant function defined on the collection $\hH(X)$ of all hyperplanes of $X$. Since $\Gamma$ has finitely many orbits on $\hH(X)$, it follows that this function is uniformly bounded from above. In other words, every hyperplane of $X$ is close to $\hh$. By Lemma~\ref{lem:dim}, this implies that the entire space $X$ is contained in some bounded neighbourhood of $\hh$.} In other words $\hh$ is trivial. Thus  $\Triv(X) = \Triv(X, \Gamma)$. 

{We just showed that every large ball of $X$ intersects every trivial hyperplane. Since $X$ is locally compact, this implies that $\Triv(X)$ is indeed finite.}

{We next point out that the inclusion  $\Ess(X) \supseteq \Ess(X, \Gamma)$ is obvious.}

In order to finish the proof, it now suffices to show that a hyperplane $\hh$ such that $X(\Gamma \hh)$ is infinite and bounded, is necessarily half-essential.  {Since $X(\Gamma \hh)$ is infinite, the hyperplane $\hh$ cannot be $\Gamma$-trivial. Therefore Lemma~\ref{lem:HalfEssential} implies that one of the two half-spaces associated with $\hh$, say $\h$, is $\Gamma$-shallow. We conclude that $\hh$ is half-essential, as desired.} 
%
\end{proof}

The following observation is sometimes useful. 

\begin{remark}\label{rem:RecognizingEssential}
Let $X$ be a CAT(0) cube complex such that $\Aut(X)$ has finitely many orbits of hyperplanes. Then there exists $N>0$ such that a hyperplane $\hh$ of $X$ is essential if and only if both $\h$ and $\h^*$ properly contain
pencils of disjoint hyperplanes of size $N$.
\end{remark}

Propositions~\ref{prop:pruning} and~\ref{prop:essential:bis} show that a few things become simpler in the presence of a  group action which is cofinite on the hyperplanes. The most obvious example of such an action is when $\Gamma \leq \Aut(X)$ is cocompact. More generally, we have the following. 

\begin{lem}\label{lem:fg}
Let $\Gamma$ be a finitely generated group acting properly discontinuously on a locally compact \cat cube complex $X$. Assume there exists a vertex $v$ such that, for each half-space $\h$, the orbit $\Gamma.v$ meets both $\h$ and $\h^*$. 

Then $\Gamma$ has finitely many orbits of hyperplanes. 
\end{lem}
\begin{proof}
The group $\Gamma$ is finitely generated if and only if every orbit has some  tubular neighbourhood which is pathwise connected. The result follows by applying this to $\Gamma.v$.  
\end{proof}

\begin{remark}
The condition on the vertex $v$ appearing in the statement of Lemma~\ref{lem:fg} is a \textbf{minimality condition} on the $\Gamma$-action. The Lemma could be {generalised in the following way: \emph{Let $\Gamma$ be a finitely generated group acting properly discontinuously on a locally compact \cat cube complex $X$. Then for any given vertex $v$, the group $\Gamma$ has finitely many orbits of hyperplanes $\hh$ such that the half-spaces $\h$ and $\h^*$ both contain orbit points from $\Gamma.v$}. }
\end{remark}

\begin{remark}\label{rem:WiseExample}
{An example, which was communicated to us by Dani Wise, shows that  the converse of Lemma~\ref{lem:fg} does not hold. Indeed, a countably based free group $F_\infty$ can act properly discontinuously on a \cat cube complex with only three orbits of hyperplanes. }

Here is a description of this example (see Figure~\ref{fig:Wise}). One starts with a graph consisting of a chain of bigons  --- take two copies $\RR_1, \RR_2$ of the real line, glued together along the integers.
One now attaches two infinite rectangular strips. The first is attached at top and bottom to $\RR_1, \RR_2$ in a straightforward manner.
The second is attached with a unit translation on one side. In this way one obtains a locally \cat cube complex with exactly three hyperplanes and whose fundamental group is $F_\infty$. 

\begin{center}
\begin{figure}
\includegraphics{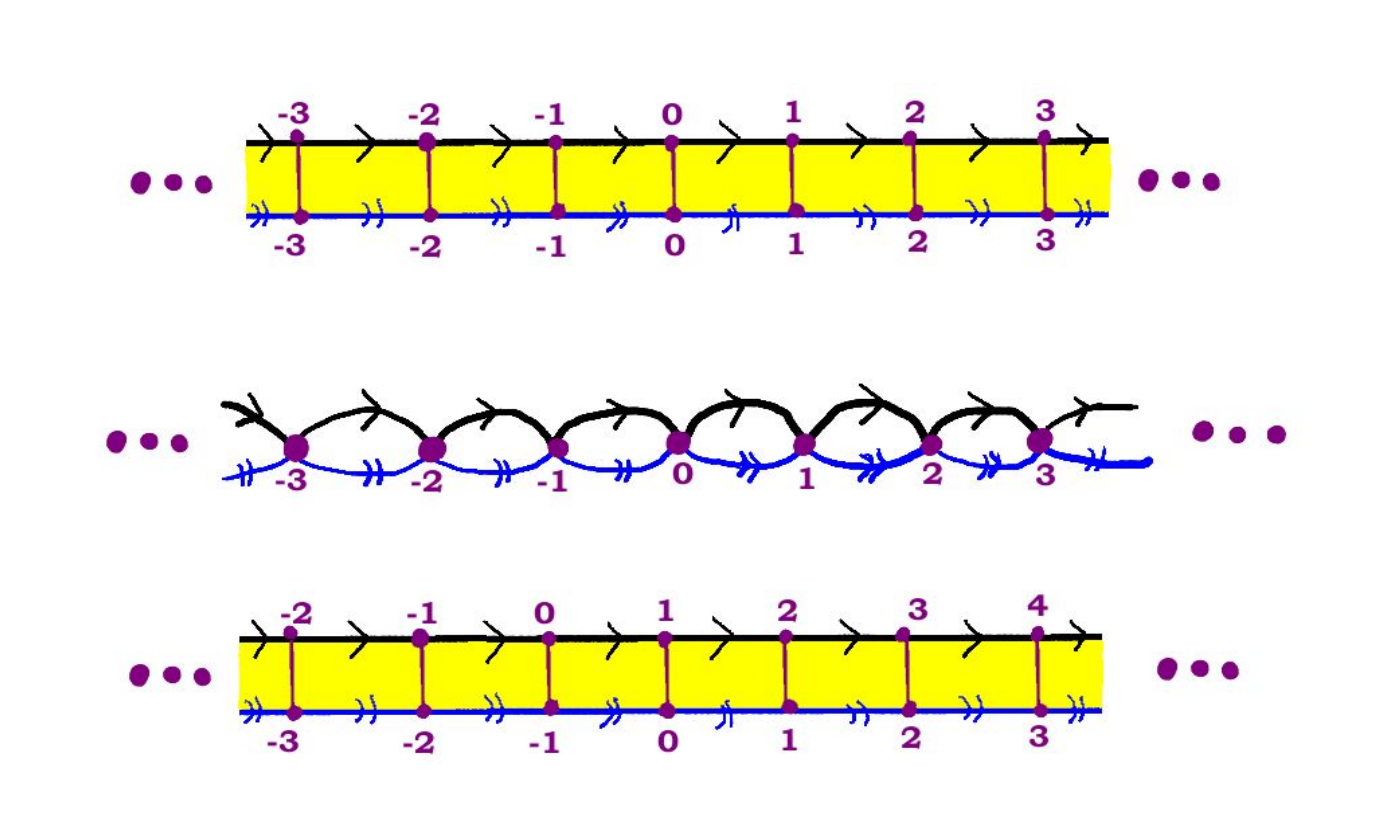}
\caption{A proper action of $F_\infty$ with three orbits of hyperplanes}\label{fig:Wise}
\end{figure}
\end{center}

\end{remark}

The following summarizes some of the results obtained thus far in the case of a finitely generated discrete group acting properly.

\begin{prop}\label{prop:fg}
Let $\Gamma$ be a finitely generated group acting properly discontinuously on a finite-dimensional locally finite \cat cube complex $X$. Let $Y$ denote the $\Gamma$-essential core of $X$. Then
\begin{enumerate}[(i)]
\item $\Gamma$ has finitely many orbits on $\hH(Y) = \Ess(X, \Gamma)$.

\item $Y$ is unbounded if and only if $\Gamma$ has no global fixed point in $X$. 

\item $Y$ embeds as a convex $\Gamma$-invariant subcomplex of $X$.

\item Every hyperplane of $Y$ is skewered by some element of $\Gamma$. 
\end{enumerate}
\end{prop}

\begin{proof}
Pick any vertex $v$. Let $Z$ denote the convex subcomplex of $X$ whose $0$-skeleton is the set of all vertices contained in every half-space containing entirely the orbit $\Gamma.v$. Thus $Z$ is non-empty, $\Gamma$-invariant, and for each hyperplane $\hh \in \hH(Z)$, the orbit $\Gamma.v$ meets both $\h$ and $\h^*$. By Lemma~\ref{lem:fg}, this implies that $\Gamma$ has finitely many orbits on $\hH(Z)$. By Lemma~\ref{lem:Haglund},  we have $\Ess(X, \Gamma) \subseteq \hH(Z)$. The assertion (i) follows. 

We can now run the pruning process described in Proposition~\ref{prop:pruning}. This proves the assertions (ii) and (iii). Assertion~(iv) follows from Lemma~\ref{lem:WallOrbit} and Proposition~\ref{prop:essential}.
\end{proof}

\section{The Flipping Lemma }\label{sec:Flipping}

{The goal of this section is to provide a proof of the Flipping Lemma, stated in the introduction. In fact, we shall give two distinct (and conceptually different) versions of the Flipping Lemma, respectively in Theorems~\ref{thm:Flipping} and~\ref{thm:Flipping:cocpt}, with different proofs and slightly different statements. The first one is shorter and easier; moreover it applies to \cat cube complex that are  possibly not locally compact. But this first proof is not self-contained as it relies on the main theorem from~\cite{CapraceLytchak}. The second is proof is self-contained and uses only the combinatorics of hyperplanes.  However it is technically much more involved; its hypotheses do not require the absence of fixed point at infinity but do require the ambient \cat space to be locally compact with a cocompact automorphism group; its conclusions are also more precise, as it yields a product decomposition of the ambient cube complex with a linelike factor. All in all, this second proof is probably less transparent, and a reader who is willing to rely on \cite{CapraceLytchak} might want to skip it. }

\medskip
Let $X$ be a \cat cube complex. Given a pair of half-spaces $\h$ and $\k$ associated to disjoint hyperplanes, we say that they 
are \textbf{nested} if $\h\subset\k$ or $\k\subset\h$. If $\h$ and $\k$ are not nested and $\h\cap\k\not=\varnothing$, we say that $\h$ and $\k$ are \textbf{facing}.  An isometry $g\in \Aut(X)$ is said to \textbf{flip} $\h$ if $\h$ and $g\h$ are facing. Equivalently $\h^* \subsetneq g\h$.

Given a group $\Gamma \leq \Aut(X)$, we say that $\h$ is \textbf{$\Gamma$-flippable} if there is an element of $\Gamma$ which flips $\h$. 

\subsection{Version I: no fixed point at infinity}
This section is devoted to proving the following theorem.

\begin{thm}[Flipping Lemma, version I]\label{thm:Flipping}
Assume that $X$ is finite-dimensional {and let $\Gamma \leq \Aut(X)$ be any subgroup}. Let $\h$ be a half-space which is not $\Gamma$-flippable. Then  $\Gamma$ has a fixed point in the visual boundary $\bd \h^*$ or $\h$ is $\Gamma$-shallow. 
\end{thm}

Notice that no assumption on $\Gamma$ or on its action is made. Moreover, we do not require $X$ to be locally compact.

\medskip
It is a well known basic fact that any finite collection of pairwise crossing hyperplanes in a \cat cube complex has a non-empty intersection, see \cite[Theorem~4.14]{Sageev95}. 
{A more general type Helly property exists for convex subcomplexes of \cat cube complexes and was established by Gerasimov \cite{Gerasimov}, stating that the intersection of convex pairwise intersecting subcomplexes is non-empty. For completeness, we include a proof of a special case of this theorem, suited to our needs.  }

\begin{lem}\label{lem:Filter}
Let $X$ be a \cat cube complex. Let $\h_1, \dots, \h_n$ be half-spaces which have pairwise a non-empty intersection. Then the intersection $\bigcap_{i=1}^n \h_i$ is non-empty.
\end{lem}

\begin{proof}
The set $H = \{\h_1, \dots, \h_n\}$ is partially ordered by inclusion. Upon renumbering, we may assume that the first $m$ half-spaces $\h_1, \dots, \h_m$ are the minimal elements of $H$. In other words, for each $i >m$, there is some $i' \leq m$ such that $\h_{i'} \subset \h_i$. In particular it suffices to show that $\bigcap_{i=1}^m \h_i$ is non-empty. {We shall show by induction on $m$ that  $\bigcap_{i=1}^m \h_i$ contains a vertex.} 

If the $\hh_i$'s are pairwise disjoint, then we have $\hh_i \subset \h_j$ for all $i \neq j$ and thus any vertex contained in $\h_i$ and adjacent to $\hh_i$ is also contained in every other $\h_j$. 

Otherwise some $\hh_i$ crosses some $\hh_{i'}$. {We view $\hh_{i'}$ as a convex subcomplex of the first cubical subdivision of $X$. In this way, the hyperplane $\hh_{i'}$ becomes a \cat cube complex whose vertices are midpoints of edges of $X$. By induction, the intersection of all $\h_j$'s such that $\hh_i$ crosses $\hh_{i'}$ contains some vertex of $\hh_{i'}$.  This vertex is the midpoint of an edge of $X$ transverse to $\hh_{i'}$. We let $v$ denote the vertex of that edge contained in $\h_{i'}$. By contruction $v$ is contained $\h_j$ for all $j$ such that $\hh_j$ crosses $\hh_{i'}$. Since $\hh_{i'}$ is entirely contained in every $\h_j$ such that $\hh_j$ does not cross $\hh_{i'}$, it follows that $v$ is in fact contained in all the $\h_i$'s.}
\end{proof}

We are now ready for the following. 

\begin{proof}[Proof of Theorem~\ref{thm:Flipping}]
Consider the $\Gamma$-orbit  $\{\gamma.\h^*\}_{\gamma \in \Gamma}$. The fact that $\h$ is $\Gamma$-unflippable precisely means that any two half-spaces in that orbit have a non-empty intersection. In view of Lemma~\ref{lem:Filter}, we deduce that any finite set of half-spaces from that orbit has a non-empty intersection. We are thus in a position to invoke Proposition~\ref{prop:CL}. This implies that either $\bigcap_{\gamma \in \Gamma} \gamma.\bd \h^*$ is a non-empty subset of the visual boundary which contains a $\Gamma$-fixed point, or the set $Y = \bigcap_{\gamma \in \Gamma} \gamma. \h^*$ is a non-empty $\Gamma$-invariant subspace of $X$, {which we assume from now on}. Since $\h$ is disjoint from $Y$, Lemma~\ref{lem:Haglund} implies that   $\h$ is $\Gamma$-shallow, as desired.
\end{proof}

Our next goal is to establish a slightly refined version of the Flipping Lemma in the case of cocompact  actions. We first need to take a detour to some additional basic considerations. 

\subsection{Non-skewering behaviours}

Let $X$ be a \cat cube complex, $\gamma \in \Aut(X)$ be a hyperbolic isometry and $\hh $ be a hyperplane. In Section~\ref{sec:prel}, we defined what it means for $\gamma$ to skewer $\hh$. There are also two types of non-skewering behaviours, which we describe as follows.

We say that $\gamma$ is  \textbf{parallel} to $\hh$ if an axis for $\gamma$ (and hence all axes of $\gamma$) lies in a neighborhood of $\hh$. If $\gamma$ does not skewer $\hh$ and is not parallel to $\hh$, we say that it is \textbf{peripheral} to $\hh$.

\begin{lem}\label{lem:Parallel} A hyperbolic isometry $\gamma$ is parallel to $\hh$ if and only if 
$\hh$ contains a geodesic line which is at bounded Hausdorff distance from some $\gamma$-axis. 

If in addition $X$ is locally compact, then $\gamma$ is parallel to $\hh$ if and only if there exists some $n>0$ such that $\gamma^n\in \Stab(\hh)$
\end{lem}

\begin{proof}
If $\hh$ contains a geodesic line which is at bounded Hausdorff distance from some $\gamma$-axis, then this axis is contained in a bounded neighbourhood of $\hh$. Thus $\gamma$ is parallel to $\hh$ as required. Conversely, if $\gamma$ has some axis $\ell$ in a bounded neighbourhood of $\hh$, then   the two endpoints of any of its axes are contained in the visual boundary $\bd \hh$. Since $\hh$ is {closed and} convex, it easily follows that it must thus contain some geodesic line joining them (see for example \cite[Prop.~3.6]{CapraceMonod1}), and the desired assertion follows.  

Assume now that $X$ is locally compact. If $\gamma\in \Stab(\hh)$ then there exists an axis of $\gamma$ in $\hh$ and we are done. Conversely, suppose that an axis $\ell$ for $\gamma$ {is contained in the $R$-neighbourhood} $\mathcal N_R(\hh)$ for some $R>0$. Consider some $p$ along $\ell$. Then for any $n>0$, we have that $\gamma^n(\hh)\cap B_R(p)\not=\varnothing$. By local finiteness, {only finitely many hyperplanes  meet the ball} $B_R(p)$, so by the pigeonhole principle, $\gamma^n(\hh)=\gamma^m(\hh)$ for some $n>m>0$. It follows that $\gamma^{n-m}(\hh)=\hh$, as required. 
\end{proof}

\begin{lem}\label{lem:Peripheral} 
{Assume that $X$ is finite-dimensional. } 
If a hyperbolic isometry $\gamma$ is peripheral to $\hh$, then there exists some $n>0$ such that $\gamma^n\hh\cap\hh=\varnothing$.
\end{lem}
\begin{proof}
{The elements of the set $\{\gamma^n\hh\vert n>0\}$ are pairwise distinct, since otherwise some positive power of $\gamma$ would stabilise $\hh$ and, hence, every $\gamma$-axis would lie in a bounded neighbourhood of $\hh$. Since $X$ is finite-dimensional,} we infer that  there exists $n>m>0$ such that $\gamma^n\hh\cap \gamma^m\hh=\varnothing$. It follows that $\gamma^{n-m}\hh\cap\hh=\varnothing$, as required. 
\end{proof}

\subsection{Endometries of proper metric spaces}\label{sec:Endometries}

Given a metric space $X$, an \textbf{endometry} of $X$ is an
injective map $\alpha : X \to X$ which is an isometry onto its image.
The following general fact is independent of the \cat inequality.

\begin{prop}\label{prop:endom}
Let $X$ be a proper metric space with a cocompact isometry group.
Then every endometry $\alpha : X \to X$ is surjective.
\end{prop}

\begin{proof}
Given $\vareps >0$, a subset $E \subset X $ is called
\textbf{$\vareps$-separated } if any two distinct points of $E$ are at
distance at least $\vareps$.

Let $(\vareps_n)_{n \geq 0}$ be a sequence of positive reals tending
to $0$ as $n$ tends to infinity.  For each $n$, let $C_n$ be the maximal cardinality
of a $\vareps_n$-separated subset contained in a ball of radius $n$
of $X$. {Notice that $C_n$ is finite since $X$ is proper and has a cocompact isometry group.} Let also $x_n \in X$ be such that the ball $B(x_n, n)$ of
radius $n$ centred at $x_n$ contains a $\vareps_n$-separated subset,
say $E_n$, of cardinality $C_n$.

Since $X$ has cocompact isometry group, there is no
loss of generality in assuming that the sequence $(x_n)$ converges
to some $x \in X$. Let now $y \in X$ be any point and $r = d(y,
\alpha(x))$. Then for each $n$ which is sufficiently larger than
$r$, there is some $y_n \in E_n$ such that $d(y, \alpha(y_n)) \leq
\vareps_n$. In particular the sequence $(\alpha(y_n))$ converges to
$y$. Since $X$ is proper, it is complete and so is $\alpha(X)$. In
particular $\alpha(X)$ is closed in $X$ and we deduce that $y \in
\alpha(X)$. Thus $\alpha(X) = X$, as desired.
\end{proof}

{Notice that the conclusion of Proposition~\ref{prop:endom} fails if $\Isom(X)$ does not act cocompactly, as illustrated by the example of $X = \RR^+$. 
}

\medskip
 Given a hyperplane $\hh$ in a cube complex $X$, we denote by $\Ess(\hh)$ the set of all hyperplanes crossing $\hh$ and which are essential in $\hh$. Endometries will naturally occur through the following. 

\begin{lem}\label{lem:Ess:endometry}
Let $X$ be a locally compact \cat cube complex with cocompact automorphism group, let $\hh, \hk$ be hyperplanes and let $X_{\hh}, X_{\hk}$ denote their respective essential cores. Assume that every hyperplane in $\Ess(\hh)$ crosses $\hk$. Then there is an isometric embedding { $X_{\hh} \to X_{\hk}$.}
\end{lem}

\begin{proof}
We shall define a map $X_{\hh} \to X_{\hk}$ at the level of the $1$-skeletons. For each evertex $v  \in X_{\hh}^{(0)}$, we need to associate a unique vertex $v' \in X_{\hk}^{(0)}$. To this end, we associate to $v$ an \textbf{ultrafilter} $\phi_v$ on $\Ess(\hk)$, \emph{i.e.} a function which chooses one side of each hyperplane in $\Ess(\hk)$. We shall do this in such a way that the intersection of all half-spaces $\bigcap_{\hat \g \in\Ess(\hk)}\phi_v(\hat \g)$ is non-empty. This intersection must therefore contain a single vertex of $X_{\hk}$. This will be our definition of $v'$. 

In order to define the ultrafilter $\phi_v$, we proceed as follows. Let $\hat \g \in \Ess(\hk)$. If $\hat \g$ does not cross $\hh$, we define $\phi_v(\hat\g)$ so that it contains $\hh$. If  $\hat \g \in \Ess(\hh)$,  we define $\phi_v(\hat\g)$ so that it contains $v$. If  $\hat \g \in \Hess(\hh)$,  we define $\phi_v(\hat\g)$ so that it contains the deep side of $\hh$ determined by $\hat \g$. It remains to define  $\phi_v(\hat\g)$ in case  $\hat \g \in \Triv(\hh)$. To do this, we choose arbitrarily a point $p \in \hh$ which does not lie on any hyperplane crossing $\hh$, and we simply define  $\phi_v(\hat\g)$ in such a way that it contains $p$. 

It is now easy to see that $\phi_v$ is an ultrafilter that satisfies the Descending Chain Condition.  {Since $X$ is finite dimensional, it follows that $\phi_v$ determines a vertex of $X_{\hk}$ (see  Section \ref{Pocsets}.)}


This defines a map $X_{\hh}^{(0)} \to X_{\hk}^{(0)}$. This map preserves the relation of adjacency: Indeed, two distinct vertices of $X_{\hh}$ are adjacent if and only if they are separated by a unique hyperplane, and the definition of $\phi_v$ is so that this will then be the case of the corresponding vertices of $X_{\hk}$. Thus we have a simplicial isometric embedding $X_{\hh}^{(1)} \to X_{\hk}^{(1)}$. In view of the way higher dimension cubes of $X_{\hh}$ and $ X_{\hk}$ are defined, the existence of {an} isometric embedding follows.
%
\end{proof}

\subsection{Version II: hereditarily essential actions}\label{sec:R-like}

{The goal of this section is to present a proof of the Flipping Lemma which does rely on Proposition~\ref{prop:CL}. The argument is self-contained but technically more involved; moreover, it works only in the special case when $X$ is locally compact and has a cocompact automorphism group. However, this approach has  the advantage that it also works when the group $\Gamma$ under consideration is allowed to fix points in the visual boundary of the ambient cube complex. This advantage will have some relevance for the applications of Rank Rigidity we shall present later, but the reader who is willing to exclude the existence of $\Gamma$-fixed points at infinity in $X$ from the start should skip this section.}

We need two more definitions. 


We say that a \cat cube complex $X$ is \textbf{$\RR$-like} if there is an $\Aut(X)$-invariant geodesic line $\ell \subset X$; this is a special case of the condition that appeared in Theorem~\ref{thmintro:PseudoEucl}. If in addition $\Aut(X)$ acts cocompactly, then $X$ is quasi-isometric to the real line $\RR$. 

Given a group $\Gamma \leq \Aut(X)$ we say that the $\Gamma$-action on $X$ is \textbf{hereditarily essential} if $\Gamma$ acts essentially and if for any finite collection of pairwise crossing hyperplanes $\hh_1, \dots, \hh_n$, we have 
$$\Ess\big(\hh_1 \cap \dots \cap \hh_n\big) =\Ess\big(\hh_1 \cap \dots \cap \hh_n, \Stab_\Gamma(\hh_1 \cap \dots \cap \hh_n)\big).$$
In view of Proposition~\ref{prop:essential}, this amounts to requiring that for any hyperplane $\hh_0$ 
 such that $\hh_0 \cap \hh_1 \cap \dots \cap \hh_n$ is an essential hyperplane of the cube complex $\hh_1 \cap \dots \cap \hh_n$, there is an element in $\Stab_\Gamma(\hh_1 \cap \dots \cap \hh_n)$ which skewers $\hh_0$. 

Basic examples of such actions are provided by groups acting cocompactly on locally compact \cat cube complex; {this is the main example that the reader should keep in mind at a first reading of Theorem~\ref{thm:Flipping:cocpt} below.} Other relevant examples will be provided by Corollary~\ref{cor:HereditarilyEss} below. 

\begin{thm}[Flipping Lemma, version II]\label{thm:Flipping:cocpt}
Let $X$ be a locally compact unbounded \cat cube complex with a cocompact automorphism group and $G \leq \Aut(X)$ be a (possibly non-closed) subgroup whose action is hereditarily essential. Let also  $\h$ be a half-space which is unflippable by the  action of $G$. 

Then $X$ has a decomposition $X=X_1\times X_2$ into a product of subcomplexes, corresponding to a transverse hyperplane decomposition $\hH(X)=\hH_1 \cup \hH_2$, which satisfies the following properties.
\begin{enumerate}[(i)]
\item $X_1$ is irreducible and all of its hyperplanes are compact. 
\item Some finite index subgroup $G' \leq G$ preserves the decomposition $X=X_1\times X_2$.
\item The $G'$-orbit of $\hh$ is in $\hH_1$. 

\item $G'$ fixes a point in the visual boundary $\bd X_1$. 

\item If in addition $G$ is closed, unimodular and acts cocompactly, then $X_1$ is $\RR$-like.
\end{enumerate} 
\end{thm}

{
We start by proving the following basic special case of Theorem~\ref{thm:Flipping:cocpt}, which will be used in the proof of the general case. 
\begin{lem}\label{lem:CompactHyperplanes}
Let $X$ be a locally compact unbounded \cat cube complex with a cocompact automorphism group and $G \leq \Aut(X)$ be a (possibly non-closed) subgroup acting essentially.  Let also  $\h$ be a half-space which is unflippable by the  action of $G$. 

If all hyperplanes of $X$ are compact, then $G$ fixes a point in the visual boundary  $\bd X$.  If in addition $G$ is closed, unimodular and acts cocompactly on $X$, then $X$ is $\RR$-like.
\end{lem}

Notice that the assumption that $G$ be unimodular is necessary. Indeed, consider the regular trivalent tree $T$ and let $G \leq \Aut(T)$ be the stabiliser of an end $\xi \in \partial T$. Then $G$ is closed, acts cocompactly, and no half-tree containing $\xi$ is $G$-flippable. Clearly $T$ is not $\RR$-like. The point is of course that this group $G$ is not unimodular. }

\begin{proof}[Proof of Lemma~\ref{lem:CompactHyperplanes}]
{Since all hyperplanes are compact, the space $X$ cannot contain any $2$-dimensional flat. Therefore $X$ must be Gromov hyperbolic (see \cite[Theorem~III.H.1.5]{Bridson-Haefliger}). 

Next we consider the orbit $\{g.\h \; | \; g \in G\}$. Since $\h$ is not $G$-flippable, any two half-spaces in this orbit have a non-empty intersection. Therefore any finite subset of $\{g.\h \; | \; g \in G\}$ has a non-empty intersection by Lemma~\ref{lem:Filter}. On the other hand, the intersection $\bigcap \{g.\h \; | \; g \in G\}$ is a proper $G$-invariant convex subcomplex, and must therefore be empty. Since $X$ is proper, we deduce that $\bigcap \{\bd g.\h \; | \; g \in G\}$ is a non-empty $G$-invariant subset of $\bd X$. 

We claim that it is reduced to a single point. Indeed, if it contained two distinct points $\xi, \eta$, then these points would be joined by some geodesic line $\ell$ since $X$ is Gromov hyperbolic and has a cocompact isometry group. In fact the set of all geodesic lines joining $\xi$ to $\eta$ is of the form $\ell \times K$ for some compact set $K$. The condition that $\xi, \eta $ both belong to $\bd g.\h$ implies that $g.\h$ intersects $\ell \times K$ in some subset of the form $\ell \times K'$ for some closed convex subset $K' \subseteq K$. Since $K$ is compact, it follows that the intersection  $\bigcap \{\bd g.\h \; | \; g \in G\}$ is non-empty, a contradiction. The claim stands proven. 

The claim readily implies that $G$ fixes a point  $\xi \in \bd X$. It remains to prove that if $G$ is closed, unimodular and acts cocompactly on $X$, then $X$ is $\RR$-like. We shall content ourselves by providing a direct argument in the special case when $G$ is discrete. In the general case, we refer the interested readed to  \cite{CapraceMonod3} or \cite[Theorem~3.14]{CapraceMonod2}, which provide general statements on fixed points at infinity for unimodular groups acting cocompactly by isometries on \cat spaces. 

Assume thus that $G$ is discrete and acts cocompactly. Since $G$ fixes $\xi$, it preserves the collection of  horoballs centred at $\xi$. The $G$-action being cocompact and essential, there must exist some $g \in G$ which permute these horoballs non-trivial. This element $g$ must therefore be a hyperbolic isometry having $\xi$ as one of its two unique fixed points in $\bd X$. Let $\eta \in \bd X$ be the other $g$-fixed point. We claim that $G$ fixes $\eta$. Otherwise some $h \in G$ would move it, which is impossible because the subgroup generated by $g$ and the conjugate $h g h\inv$ would then be non-discrete. In particular $G$ preserves the collection of geodesic lines joining $\xi$ to $\eta$. As $G$ acts cocompact, this forces $X$ to be $\RR$-like, as desired.
}
\end{proof}

\begin{proof}[Proof of Theorem~\ref{thm:Flipping:cocpt}]
We are given a non-flippable half-space $\h$ and need to show that $X$ decomposes as a product. We divide $\hH = \hH(X)$ into two subsets:

$$\hH_1=\{\hk\in\hH\vert \Ess(\hk)=\Ess(\hh)\}$$
and 
$$\hH_2=\hH-\hH_1.$$

We aim to show that every hyperplane in $\hH_1$ intersects every hyperplane in $\hH_2$. This will give a product decomposition of $X$ by Lemma~\ref{lem:Prod}.
We proceed in several steps. 

\begin{claim}
For all $\hk \in \hH$ such that $\hk \subset \h$, we have $\Ess(\hk)\subset\Ess(\hh)$.
\end{claim}

Note that since $\h$ is unflippable, it follows that every element with an axis in $\h$ is parallel to $\hh$, see Lemma~\ref{lem:Peripheral}. Consequently, every element of $\Stab_G(\hk)$ is parallel to $\hh$. This immediately tells us that $\Ess(\hk)\subset\Ess(\hh)$.  For if there were an essential hyperplane in $\hk$ disjoint from $\hh$, we would simply take an element $\alpha\in\Stab_G(\hk)$ such that $\alpha$ skewers  that hyperplane and then $\alpha$ would have an axis staying arbitrarily far away from $\hh$. 

\begin{claim}
For all $g \in G$ such that  $g\hh\subset\h$, we have  $\Ess(g\hh)= \Ess(\hh)$.
\end{claim}

We already have that $\Ess(g\hh)\subset\Ess(\hh)$ by the first claim. To show equality, consider the essential cores $X_{g\hh}$ of $g\hh$ and $X_{\hh}$ of $\hh$. The fact that there is an inclusion of $\Ess(g\hh)$ into $\Ess(\hh)$ tells us that there is an isometric embedding $X_{g\hh}\to X_{\hh}$, see Lemma~\ref{lem:Ess:endometry}. But these two spaces are isometric and have a cocompact automorphism group since $X$ is locally compact and $\Aut(X)$ acts cocompactly. Therefore, by Proposition~\ref{prop:endom}, we have that $X_{g\hh}$ and $X_{\hh}$ are isomorphic as cube complexes. Thus $\Ess(g\hh)=\Ess(\hh)$.  Note that this means that for any $g_1, g_2$ with $g_1\hh \subset g_2\h$, we have that $\Ess(g_1\hh)=\Ess(g_2\hh)$.

\begin{claim}
For all $\hk \in \hH$ such that $\hk \subset \h$, we have $\Ess(\hk)=\Ess(\hh)$.
\end{claim}

To see this, consider some element 
$g$ skewering $\hk$. Since $\hk$ is disoint from $\hh$, the element $g$ must also skewer $\hh$, otherwise $g$ would be peripheral to $\hh$ and some power of it would therefore flip $\h$ (see Lemma~\ref{lem:Peripheral}) contradicting our assumption. After possibly replacing $g$ by some power of $g$, we may assume that $g\h\subset\h$. If for every positive power of $g$, we have that $g^n\hh\cap\hk\not=\varnothing$, then {by Remark~\ref{rem:RecognizingEssential}} we would have that for some positive power of $g$, say $g^m$, $g^m\hh\in\Ess(\hk)$, which contradicts our previous claim that $\Ess(\hk)\subset\Ess(\hh)$. Thus, by passing to an appropriate positive power of $g$, we may assume that $g\h\subset\k$.  Thus $\hk$ separates $\hh$ from $g\hh$. But we have already established that $\Ess(g\hh)=\Ess(\hh)$. In particular $\Ess(\hh)\subset\Ess(g\hh)$. Thus 
$\Ess(\hh)\subset\Ess(\hk)$. Note that this means that for any $g\in G$ with $\hk\subset g\h$, we have that $\Ess(\hk)=\Ess(g\hh)$.

\begin{claim}
For all $\hk \in \hH$ such that $\hk \cap \hh = \varnothing$, we have $\Ess(\hk)=\Ess(\hh)$.
\end{claim}

We have already proved the claim for all hyperplanes $\hk$ contained in $\h$.  We need to now prove the claim for all hyperplanes $\hk\subset\h^*$. Let $\k$ denote the half-space associated to $\hk$ such that $\h\subset\k$. Consider some $g$ skewering $\hh$ so that $\h\subset g\h$. After passing to some positive power of $g$, we may assume that $\hk\not\subset g\h^*$. If $\hk \subset g\h$, then by what we have shown so far, we have that $\Ess(\hk)=\Ess(g\hh)$ and $\Ess(g\hh)=\Ess(\hh)$, so we are done. Thus, we may assume that for all positive powers of $g$, $g^n\hh\cap\hk\not=\varnothing$.  By Remark~\ref{rem:RecognizingEssential},  we may pass to some positive power of $g$ such that $g\hh$ is essential in $\hk$. Since $\Ess(g\hh)=\Ess(\hh)$, we know that $\hk$ is not essential in $g\hh$. So there are two posibilities: either there exists an $R>0$ such that $g\hh\subset \mathcal N_R(\k)$, or there exists some $R>0$ such that 
$g\hh\subset \mathcal N_R(\k^*)$. 

First, suppose that $g\hh\subset \mathcal  N_R(\k)$. Choose some hyperbolic element $a$ skewering $\hk$, so that $\k\subset a\k$. Now by passing to a high enough power of $a$, we may assume that $g\hh\cap a\hk=\varnothing.$ Since $g\hh$ is essential in $\hk$, we may find an element $b\in\Stab_G(\hk)$ such that $b$ skewers $g\hh$ so that $bg\h\subset g\h$. We claim that after passing to some positive power of $b$,  we have that $ba\hk\subset g\h$. To see 
this consider some geodesic edge path $\alpha$ from $a\hk$ to $\hk$. Let $p$ denote the terminal endpoint $\alpha$ adjacent to $\hk$. We may apply {some power of} $b$ so that $bp\in g\h$. 
Now $b\alpha$ is a path from $ba\hk$ to $bp$ which does not cross $g\hh$. Thus $ba\hk\subset g\h$. Setting $c=ba$, we have that $c$ skewers $\hk$ with
{$c\k\subset \k$}.
We claim further that $c$ skewers $g\hh$. For otherwise a positive power of $c$ would flip $g\h$ {by Lemma~\ref{lem:Peripheral}} and hence there would be an element flipping $\h$ as well.  By passing to some power of $c$, we may further assume that $c\h\subset\h$. We thus have that $\Ess(cg\hh)=\Ess(g\hh)=\Ess(\hh)$. We also have that $\Ess(c\hk)=\Ess(g\hh)$. Thus $\Ess(c\hk)=\Ess(cg\hh)$, so that $\Ess(\hk)=\Ess(g\hh)=\Ess(\hh)$. 

It remains to handle the case when $g\hh\subset \mathcal N_R(\k^*)$. We may assume that $g\hh\not\subset \mathcal N_R(\k)$ for otherwise, we are in case 1. Thus there exist points of 
$g\hh\cap \mathcal \k^*$ arbitrarily far away from $\hk$. {In particular, we can find an arbitrarily large pencil of parallel hyperplanes crossing $g\hh$ and lying entirely in $\k^*$ (see Lemma~\ref{lem:dim}). Amongst them, some must belong to $\Ess(g\hh)$ by Remark~\ref{rem:RecognizingEssential}, and none of them crosses $\hh$ since they are contain in $\k^*$. This contradicts the fact that $\Ess(g\hh)=\Ess(\hh)$, thereby  concluding the proof of the claim.}

\begin{claim}
Every hyperplane in $\hH_1$ crosses every hyperplane in $\hH_2$. 
\end{claim}

Consider a hyperplane $\hk_1\in \hH_1$ and $\hk_2\in\hH_2$ and suppose that $\hk_1\cap\hk_2=\varnothing$. Let $\k_1$ and $\k_2$ be the respective halfspaces satisfying $\k_1\subset\k_2$. Choose some $g\in G$ such that $g$ skewers $\hk_1$, so that $g\k_1\subset\k_1$. 
We know that $g$ is not parallel to $\hh$, otherwise we would have that $\hk_1$ is essential in $\hh$, contradicting the fact that $\Ess(\hk_1)=\Ess(\hh)$. Consequently, $g$ either skewers $\hh$ or $g$ is peripheral to $\hh$. { In either case, we may replace $g$ by some positive power of itself} such that $g\hh\cap\hh=\varnothing$ and $g\hh\subset \k_1$. Thus $g\hh\cap\hk_2=\varnothing$. Applying $g^{-1}$ we see that $g^{-1}\hk_2$ and $g^{-1}\hh$ are disjoint from $\hh$. Thus $\Ess(g^{-1}\hh)=\Ess(g^{-1}\hk_2)=\Ess(\hh)$ by the preceding claim. Now we apply $g$ to see that $\Ess(\hk_2)=\Ess(\hh)$ contradicting the fact that $\hk_2\in\hH_2$. 

\medskip
At this point, Lemma~\ref{lem:Prod} ensures that $X$ factors as a product $X_1\times X_2$ corresponding to the partition above as $\hH=\hH_1\cup\hH_2$.  

{Notice that $\Ess(\hh) \subseteq \hH_2$. In particular, if $\hH_2=\varnothing$, then $\Ess(\hh) = \varnothing$ and, hence $\Ess(\hk) = \varnothing$ for all $\hk \in \hH$. In view of Proposition~\ref{prop:pruning}, which can be applied since  $\Aut(X)$ has finitely many orbits of hyperplanes, it follows that every hyperplane in $X = X_1$ is compact. In particular $X$ has a unique non-compact  irreducible factor. Since the $G$-action is essential, it follows that $X$ has no non-trivial compact factor, and we deduce that $X$ must thus be irreducible. Thus assertion (i) holds, and assertions (ii), (iii) are trivial in this case. Moreover assertions (iv) and (v) follow from Lemma~\ref{lem:CompactHyperplanes}.}

Assume now that there exists a hyperplane in $\hh_2\in\hH_2$. Notice that by the definition of $\hH_1$ and $\hH_2$, we have $\Ess(\hh) = \Ess(\hk) \subset \hH_2$ for any $\hk \in \hH_1$. Moreover the factor $X_1$ is isomorphic to the restriction quotient $X(\hH_1)$ (see \S\ref{sec:prel}). It follows that for every hyperplane $\hk$ of $X_1$, the set of essential hyperplanes of $\hk$ (in the complex $X_1$) is empty. {As shown in the case $\hH_2 = \varnothing$, this implies that} every hyperplane of $X_1$ is compact. Since $X_1$ is essential (because $X$ is so), it follows that $X_1$ is irreducible. Proposition~\ref{prop:deRham} thus ensures the existence of a finite index subgroup of $\Aut(X)$ which preserves the decomposition $X = X_1 \times X_2$. 

It only remains to prove (iv) and (v) for the subcomplex $X_1$. To do this there is no loss of generality in assuming $X=X_1$. In this way, we are reduced to the case $\hH_2 = \varnothing$ which has already been treated.
\end{proof}

\begin{cor}\label{cor:Flipping:cocpt}
Let $X$ be a locally compact unbounded \cat cube complex with cocompact automorphism group and $G \leq \Aut(X)$ be a group whose action is hereditarily essential. 

Then $X$ decomposes as a product $X = X_1 \times \dots \times X_p \times Y$ of subcomplexes {(possibly $p=0$ or $Y$ is reduced to a single vertex)}. This decomposition is preserved by some finite index subgroup $G' \leq G$. Moreover $G'$ has no fixed point in $\bd Y$, every half-space of $Y$ is $G'$-flippable, and  for all $i \in\{1, \dots, p\}$, every  hyperplane of $X_i$ is compact. 

If in addition $G$ is closed, unimodular and acts cocompactly, then $X_i$ is $\RR$-like for all $i$. 
\end{cor}

\begin{proof}
If every half-space of $X$ is $G$-flippable, then $G$ cannot fix any point at infinity. We can take thus take $Y =X$ and we are done in this case. 

If there is some $G$-unflippable half-space $\h$, then Theorem~\ref{thm:Flipping:cocpt} yields a splitting $X = X_1 \times X'$. The desired result then follows by considering the cube complex $X'$ and using induction on dimension. 
\end{proof}

\section{An irreducibility criterion}\label{sec:StronglySeparated}

{
\subsection{Strongly seperated hyperplanes}
The goal of this section is to establish the final ingredient needed for the proof of the Rank Rigidity Theorem. Recall that two hyperplanes $\hh_1$ and $\hh_2$ in a \cat cube complex  are called \textbf{strongly seperated} if no hyperplane $\hk$ has a non-empty intersection with both $\hh_1$ and $\hh_2$. In particular  $\hh_1$ and $\hh_2$ must be disjoint. 

\begin{prop}\label{prop:IrredCriterion}
Let $X$ be a finite-dimensional unbounded  \cat cube complex  {such that  $\Aut(X)$ acts essentially without a fixed point at infinity. } Then the following conditions are equivalent. 
\begin{enumerate}[(i)]
\item $X$ is irreducible.

\item There is a pair of strongly separated hyperplanes. 

\item For each half-space $\h$ there is a pair of half-spaces $\h_1, \h_2$ such that $\h_1 \subset \h \subset \h_2$ and the hyperplanes $\hh_1 $ and $\hh_2$ are  strongly separated. 
\end{enumerate}
\end{prop}

Notice that a general  irreducibility criterion in terms of hyperplanes is already available from Lemma~\ref{lem:Prod}. The power of Proposition~\ref{prop:IrredCriterion} is that irreducibility can be detected on a single pair of hyperplanes, rather than on a global property of the collection of all hyperplanes. 

The proof of Proposition~\ref{prop:IrredCriterion} is postponed to Section~\ref{sec:ProofIrred} below. We first need to assemble a few subsidiary lemmas. 
}


\subsection{Finding hyperplanes in cubical sectors}

Let $\hh_1$ and $\hh_2$ be a pair of intersecting hyperplanes. The complementary components of $\hh_1 \cup\hh_2$ will be called \textbf{sectors}. Note that $\hh_2$ meets $\hh_1$ in a hyperplane of $\hh_1$, so that $\hh_1$ is divided into two \textbf{half-hyperplanes}, $\hh_1\cap\h_2$ and $\hh_1\cap\h_2^*$. Similarly, $\hh_1$ subdivides $\hh_2$ into two half-hyperplanes. Thus a sector of $\hh_1$ and $\hh_2$ is bounded by two half-hyperplanes. Two sectors determined by $\hh_1$ and $\hh_2$ are called \textbf{opposite} if they do not share a common half-hyperplane.

%
%
Our main technical lemma is the following criterion ensuring the existence of hyperplanes contained in opposite sectors determined by a pair of intersecting hyperplanes.

\begin{lem}\label{lem:HypsInSector}
Let $X$ be a finite-dimensional \cat cube complex  {such that  $\Aut(X)$ acts essentially without  a fixed point at infinity. } Suppose that $X$ is irreducible.  

Then for any pair of crossing hyperplanes $\hh$ and $\hk$, there exist two disjoint hyperplanes that are respectively contained in two opposite sectors determined by  $\hh$ and $\hk$.
\end{lem}

Before undertaking the proof of Lemma~\ref{lem:HypsInSector}, we first need to establish a number of  auxiliary facts.


\begin{lem}\label{lem:HypsInSector:aux}
Let $X$ be a finite-dimensional \cat cube complex  {such that  $\Aut(X)$ acts essentially without a fixed point at infinity. }

Let $\hh, \hk$ be a pair of crossing hyperplanes. If one of the four sectors determined by  $\hh, \hk$ contains a hyperplane, then there exist two disjoint hyperplanes respectively contained in opposite sectors determined by  $\hh$ and $ \hk$.
\end{lem}

\begin{center}
\begin{figure}[h]
\includegraphics[height=8cm]{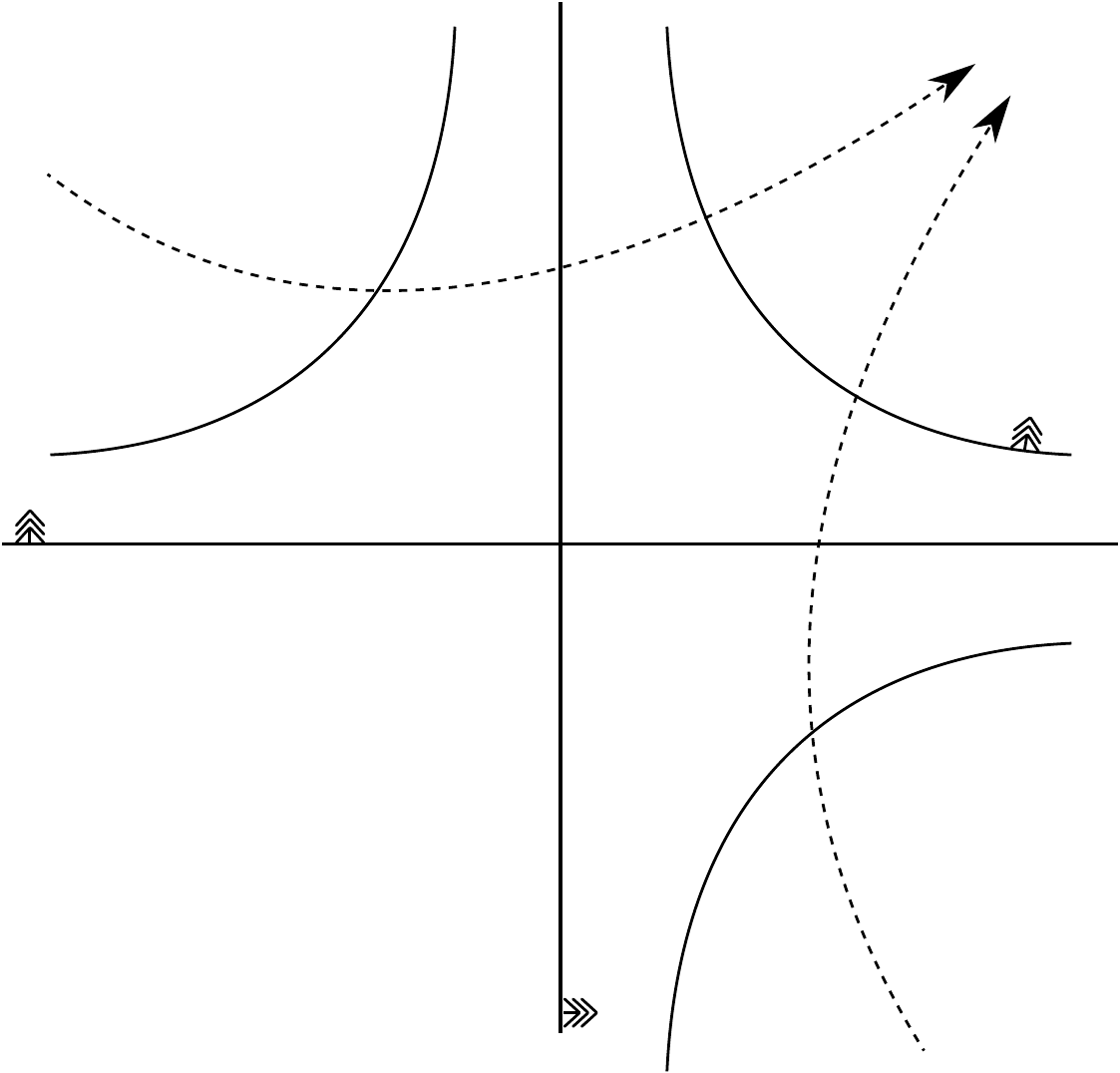}
\put(-233, 123){$\h$}
\put(-107, 10){$\k$}
\put(-20, 143){$\k'$}
\put(-30, 75){$\gamma\inv.\hk'$}
\put(-185, 212){$(\gamma')\inv.\hk'$}
\put(-23, 195){$\gamma$}
\put(-45, 215){$\gamma'$}
\caption{Proof of Lemma~\ref{lem:HypsInSector:aux}}
\label{fig1}
\end{figure}
\end{center}

\begin{proof}
Let $\hk'$ be a hyperplane contained in one of the four sectors  determined by  $\hh, \hk$. We may assume without loss of generality that $\k' \subset \h\cap \k$. Applying the Double Skewering Lemma (see \S\ref{sec:outline}), we find $g_1\in \Aut(X)$ such that $g_1.\h \subset \k'$. In particular $g_1$ skewers both $\hh$ and $\hk'$. Since $\k' \subset \k$, it follows that no $g_1$-axis is contained in a bounded neighbourhood of $\hk$. Thus $g_1$ either skewers $\hk$ or is peripheral to $\hk$.  We deduce from Lemmas~\ref{lem:def:Skewer} and~\ref{lem:Peripheral} that there exists some $n >0$ such that $g_1^n.\hk \subset \k$. Upon replacing $g_1$ by $g_1^n$, we may and shall assume that $g_1.\hk \subset \k$.

Invoking again the Double Skewering Lemma, we now find an element $g_2 \in \Aut(X)$ such that $ g_2.\k \subset g_1.\k'$. Setting $\gamma = g_2 g_1$, we notice that 
$$\gamma.\h = g_2 g_1.\h \subset g_2.\k' \subset g_2.\k \subset g_1.\k' \subset g_1.\h \subset \k',$$
from which it follows that $\gamma$ skewers both $\hh$ and $\hk'$. Moreover we also have 
$$\gamma.\hk = g_2 g_1.\hk \subset g_2.\k \subset g_1.\k' \subset \k',$$
from which it follows that either $\gamma.\k \subset \k'$, or $\gamma.\k^* \subset \k'$.

\medskip
Assume first that  $\gamma.\k \subset \k'$. Thus we have $\h \cup \k \subset \gamma\inv.\k'$. Passing to the complementary half-spaces, this means that $  \h^* \cap \k^* \supset \gamma\inv.(\k')^*$. In particular the hyperplane $\gamma\inv.\hk'$ is contained in the sector $\h^* \cap \k^*$, which is opposite to the sector $\h \cap \k$ containing the hyperplane $\hk$. Thus we are done in this case.

\medskip
Assume now that $\gamma.\k^* \subset \k'$ (see Figure~\ref{fig1}). In that case the hyperplane $\gamma\inv.\hk'$ is contained in the sector $\h^* \cap \k$. We then repeat the above construction with the roles of $\h$ and $\k$ interchanged. This yields an element $\gamma' \in \Aut(X)$ such that $\gamma'.\k \subset \k'$ and either $\gamma.\h \subset \k'$, or $\gamma.\h^* \subset \k'$. In the former case, we deduce as above that $(\gamma')\inv.\hk'$ is contained in the sector $\h^* \cap \k^*$ and we are done. Otherwise $(\gamma')\inv.\hk'$ is contained in the sector $\h \cap \k^*$. This means that we have found two hyperplanes, namely $\gamma\inv.\hk'$ and $(\gamma')\inv.\hk'$, which are respectively contained in the opposite sectors  $\h^* \cap \k$ and $\h \cap \k^*$. 
\end{proof}

We are now ready for the following.

\begin{proof}[Proof of Lemma~\ref{lem:HypsInSector}]
We denote by $\hH$ the set of all hyperplanes which are equal to or disjoint from $\hh$, and by $\hat{\mathcal K}$ the set of hyperplanes which are equal to or disjoint from $\hk$. Thus $\hh \in \hH$ and $\hk \in \hat{\mathcal K}$. Moreover, if $\hH \cap \hat{\mathcal K} \neq \varnothing$, then some hyperplane is disjoint from both $\hh$ and $\hk$, and Lemma~\ref{lem:HypsInSector:aux} then yields the desired conclusion. We assume henceforth that $\hH$ is disjoint from $\hat{\mathcal K}$. We shall consider two cases. 

\medskip
\begin{center}
\begin{figure}[h]
\includegraphics[height=8cm]{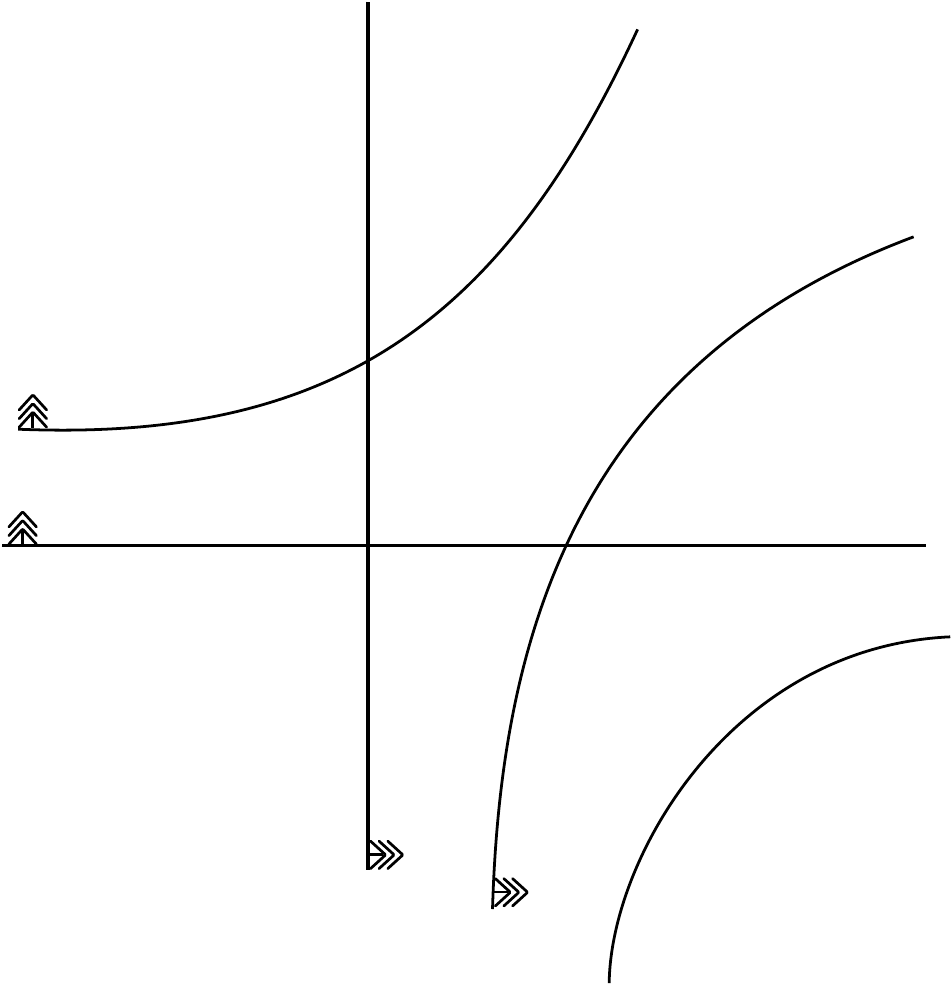}
\put(-218, 113){$\h$}
\put(-215, 140){$\h'$}
\put(-125, 28){$\k$}
\put(-97, 20){$\k'$}
\put(-20, 64){$\hh''$}
\caption{Proof of Lemma~\ref{lem:HypsInSector}}
\label{fig2}
\end{figure}
\end{center}
Assume first that there exist hyperplanes $\hh' \in \hH$ and $\hk' \in \hat{\mathcal K}$ such that $\hh'$ does not cross $\hk'$. There is no loss of generality in assuming $\h' \subset \h$ and $\k' \subset \k$ (see Figure~\ref{fig2}). 
We then apply Lemma~\ref{lem:HypsInSector:aux} to the pair of crossing hyperplanes $\hh, \hk'$ and the hyperplane $\hh'$. This implies in particular that one of the two sectors $\h \cap \k'$ or $\h^* \cap \k'$ contains some hyperplane $\hh''$. Notice that in particular we have $\hh \subset \k$. We then invoke Lemma~\ref{lem:HypsInSector:aux} once more, this time to the crossing pair $\hh, \hk$ and the hyperplane $\hh''$. This provides the required pair of hyperplanes.

\medskip
Assume next that every hyperplane in $ \hH$ meets every hyperplane in $ \hat{\mathcal K}$. We shall show that this implies that $X$ admits a non-trivial product decomposition, which contradicts the hypothesis that $X$ is irreducible. 
To this end, we define $\hH'$ (resp. $\hat{\mathcal K}'$) to be the collection of all those hyperplanes $\hat{\mathfrak a}$ which are disjoint from some hyperplane belonging to $\hH$ (resp. $\hat{\mathcal K}$). Thus we have $\hH \subseteq \hH'$ and $\hat{\mathcal K} \subseteq \hat{\mathcal K}'$. 

We claim that for each half-space $\mathfrak a$ such that $\hat{\mathfrak a} \in \hH'$, there exist two half-spaces $\mathfrak b, \mathfrak b'$ such that $\mathfrak b \subset \mathfrak a \subset \mathfrak b'$ and $\hat{\mathfrak b}, \hat{\mathfrak b}' $ both belong to $\hH$. 

Indeed, since $\hat{\mathfrak a} \in \hH'$, there is some $\hat{\mathfrak c} \in \hH$ which is disjoint from $\hat{\mathfrak a}$. Thus $\hat{\mathfrak c}$ is disjoint from $\hh$, or $\hat{\mathfrak c}= \hh$. If $\hat{\mathfrak a}$ is disjoint from $\hh$, then $\hat{\mathfrak a}$ belongs to $\hH$ and the claim follows using the fact that all hyperplanes are essential. Otherwise $\hat{\mathfrak a}$ crosses $\hh$, and we can apply Lemma~\ref{lem:HypsInSector:aux} to the crossing pair $\hat{\mathfrak a}, \hh$ and the hyperplane $\hat{\mathfrak c}$. This yields a pair of disjoint hyperplanes $\hat{\mathfrak b}$, $\hat{\mathfrak b'}$ which are separated by both $\hat{\mathfrak a}$ and $\hh$. Therefore $\hat{\mathfrak b}$, $\hat{\mathfrak b'}$ both belong to $\hH$ and the claim follows. 

Similar arguments show the corresponding claim for $\hat{\mathcal K}'$, namely: for each half-space $\mathfrak a$ such that $\hat{\mathfrak a} \in \hat{\mathcal K}'$, there exist two half-spaces $\mathfrak b, \mathfrak b'$ such that $\mathfrak b \subset \mathfrak a \subset \mathfrak b'$ and $\hat{\mathfrak b}, \hat{\mathfrak b}' $ both belong to $\hat{\mathcal K}$.

These two claims  imply that every hyperplane in $\hH'$ crosses every hyperplane in $\hat{\mathcal K}'$. Let $\mathcal R = \hH(X) \setminus (\hH' \cup \hat{\mathcal K}')$. Thus every hyperplane in $\mathcal R$ crosses every hyperplane in $\hH \cup  \hat{\mathcal K}$ and, hence, every hyperplane in $\hH' \cup \hat{\mathcal K}'$ by the above two claims. Finally, notice that the partition 
$$ \hH(X) = \hH' \cup \hat{\mathcal K}' \cup \mathcal R$$
is non-trivial since $\hh'' \in \hH'$ and $\hk \in \hat{\mathcal K}' $. Therefore, we can invoke Lemma~\ref{lem:Prod}, which yields the absurd conclusion that $X$ is a product. 
\end{proof}

We shall also need the following  variant of Lemma~\ref{lem:HypsInSector:aux}.

\begin{lem}\label{lem:HypsInSector:aux:bis}
Let $X$ be a finite-dimensional \cat cube complex  {such that  $\Aut(X)$ acts essentially without a fixed point at infinity. }

Let $\h_1 \supsetneq \h_2 \supsetneq \h_3$ be a chain of half-spaces and   $\k$ be a half-space such that $\hk$  crosses $\hh_1$,  $\hh_2$ and $\hh_3$, and the sector $\h_3 \cap \k$ contains some half-space $\k'$. 

Then there exists a pair of disjoint hyperplanes $\hh'$ and $\hh''$ that are separated by $\hk$ and at least two of the three hyperplanes $\hh_1$, $\hh_2$ and $\hh_3$.
\end{lem}

\begin{center}
\begin{figure}[h]
\includegraphics[height=8cm]{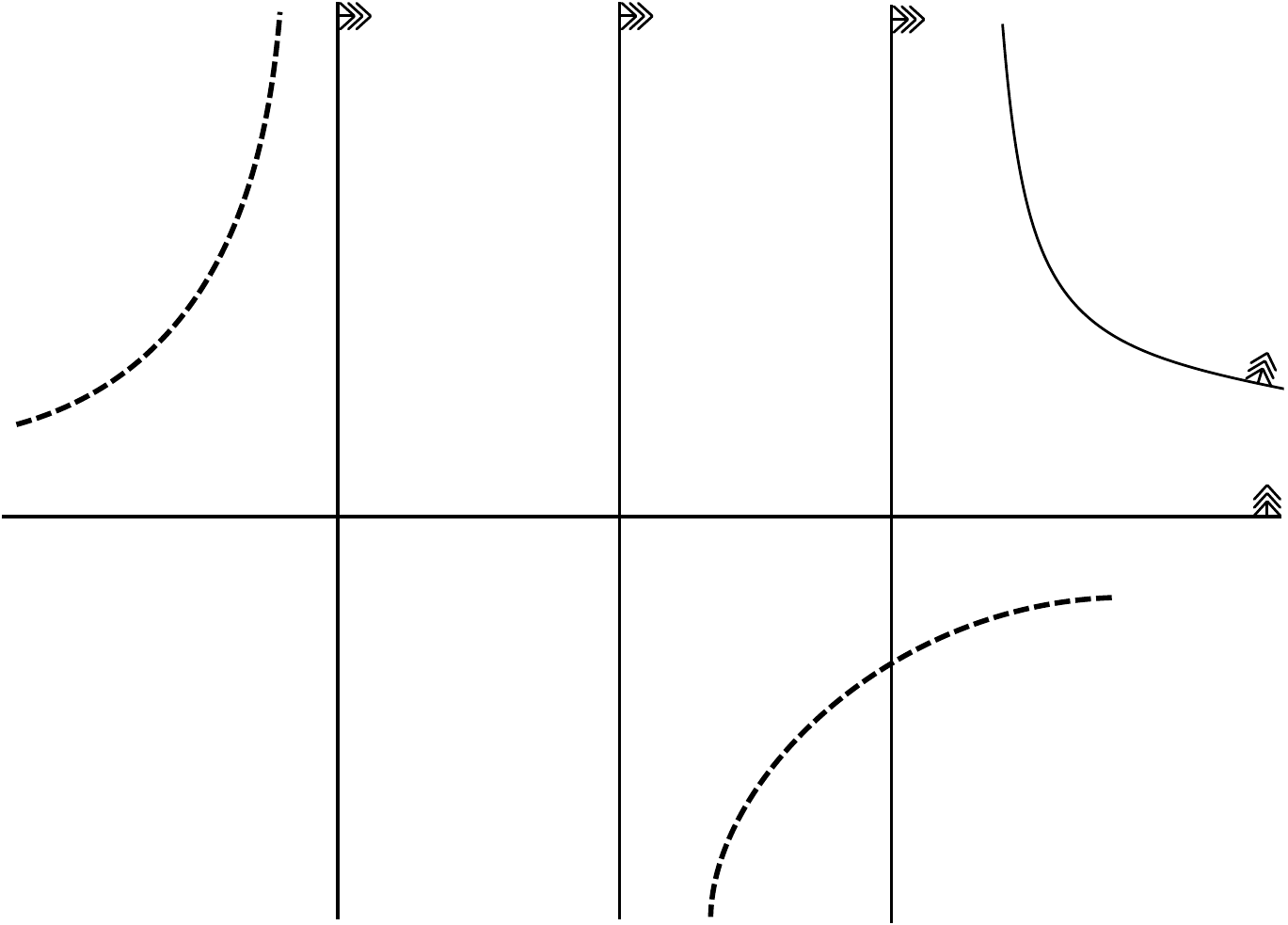}
\put(-300, 134){$\hh$}
\put(-222, 218){$\h_1$}
\put(-153, 218){$\h_2$}
\put(-88, 218){$\h_3$}
\put(-7, 144){$\k'$}
\put(-7, 112){$\k$}
\put(-130, 47){$\hh''$}
\caption{Proof of Lemma~\ref{lem:HypsInSector:aux:bis}}
\label{fig3}
\end{figure}
\end{center}
\begin{proof}
%
We first apply Lemma~\ref{lem:HypsInSector:aux} to the crossing pair $\hh_1, \hk$ and the hyperplane $\hk'$. This provides in particular a hyperplane $\hh$ disjoint from $\hh_1$ and $\hk$, which is either contained in $\h_1^* \cap \k^*$ or in $\h_1^* \cap \k$. 

\medskip
In the former case, we have two disjoint hyperplanes $\hh'= \hk'$ and $\hh' = \hh$ which are separated by each of the elements of $\{\hk, \hh_1, \hh_2, \hh_3\}$. Thus we are done in this case. 

\medskip 
Assume now that $\hh \subset \hh_1^* \cap \k$. We now invoke Lemma~\ref{lem:HypsInSector:aux} again, this time applied to the crossing pair $\hk, \hh_2$. This provides in particular a hyperplane $\hh''$ which is either contained in $\h_2 \cap \k^*$ or in $\h_2^* \cap \k^*$.

In the former case, we have two disjoint hyperplanes $\hh' = \hh$ and $\hh''$ which are separated by each of the elements of $\{\hk, \hh_1, \hh_2\}$, and we are done. 

In the latter case, we have two disjoint hyperplanes $\hh' = \hk'$ and $\hh''$ which are separated by each of the elements of $\{\hk, \hh_2, \hh_3\}$, and we are equally done. 
\end{proof}

{
\subsection{Proof of the Irreducibility Criterion}\label{sec:ProofIrred}

\begin{proof}[Proof of Proposition~\ref{prop:IrredCriterion}]
The implication (ii)$\Rightarrow$(i) follows easily from Lemma~\ref{lem:Prod}, while (iii)$\Rightarrow$(ii) is obvious. We need to show that (i)$\Rightarrow$(iii). 

Assume thus that $X$ is irreducible and let $\h$ be any half-space. Suppose for a contradiction that (iii) fails. Thus for each pair of half-spaces $\h', \h''$ such that $\h' \subsetneq \h \subsetneq \h''$, there is some hyperplane $\hk$ crossing both $\hh'$ and $\hh''$. 

Since $\Gamma$ acts essentially, there is some $\gamma \in \Gamma$ which skewers $\hh$ (see Proposition~\ref{prop:essential}). Upon replacing $\gamma$ by an appropriate power, we have $\gamma\inv\h \subsetneq \h \subsetneq \gamma\h$. Set $\h_0 = \gamma\inv\h $ and $\h'_0 = \gamma\h$. 

We shall now construct inductively an infinite sequence $(\h_n, \h'_n, \k_n)_{n>0}$ of triples of half-spaces which, together with $ \h_0 $ and $\h'_0$, satisfies the following conditions for all $n > 0$:  
\begin{enumerate}[(a)]
\item $\hk_{n}$ crosses $\hh_{n-1}$ and $  \hh'_{n-1}$.

\item $ \hk_n$ separates $\hh_n$ from $ \hh'_{n}$.

\item $\h_n \subsetneq \h_{n-1}  \subsetneq \h \subsetneq  \h'_{n-1} \subsetneq \h'_{n}$.
\end{enumerate}

For all $n>0$, we now describe an  inductive construction of the  triple $(\h_n, \h'_n, \k_n)$.  We apply the Double Skewering Lemma to the pair $\h_{n-1}  \subsetneq   \h'_{n-1}$. This yields an element $\gamma \in \Gamma$ such that $\h'_{n-1} \subsetneq   \gamma \h_{n-1} \subsetneq \gamma \h'_{n-1}$. Since the pair $\h_{n-1}, \gamma \h'_{n-1}$ cannot be strongly separated by assumption, there is some hyperplane $\hk$ crossing both  $\h_{n-1}$ and $ \gamma \h'_{n-1}$. By Lemma~\ref{lem:HypsInSector}, one of the sectors $\k \cap  \h_{n-1}$ or $\k^* \cap \h_{n-1} $ must contain properly some half-space. Upon replacing $\k$ by $\k^*$, we are thus in a position to invoke  Lemma~\ref{lem:HypsInSector:aux:bis} to the chain $\h_{n-1}  \subsetneq   \h'_{n-1}\subsetneq \gamma \h'_{n-1}$ and the half-space $\k$. This yields two hyperplanes $\hh'$ and $\hh''$ which are separated by $\hk$ and at least two of the three hyperplanes $\hh_{n-1}$, $    \hh'_{n-1}$ and $\gamma\hh'_{n-1}$. 

If $\hh'$ and $\hh''$ are separated by $\hh_{n-1}$ and $    \hh'_{n-1}$, then we define $\h_n$ as the half-space bounded by $\hh'$ and contained in $\h_{n-1}$, and we define $\h'_{n}$ as the half-space bounded by $\hh''$ and containing $\h'_{n-1}$. We also set $\k_n = \k$ in this case. 

If  $\hh'$ and $\hh''$ are separated by $\hh'_{n-1}$ and $ \gamma \hh'_{n-1}$, then the hyperplanes $\gamma\inv \hh'$ and $\gamma\inv \hh''$ are separated by $\gamma\inv \hh'_{n-1}$, $\hh_{n-1}$ and $\hh'_{n-1}$. We then define $\h_n$ as the half-space bounded by $\gamma\inv\hh'$ and contained in $\h_{n-1}$, and we define $\h'_{n}$ as the half-space bounded by $\gamma\inv \hh''$ and containing $\h'_{n-1}$. We also set $\k_n = \gamma\inv \k$ in this case. 

In either case, the hyperplane $\hk_n$ crosses both  $\h_{n-1}$ and $  \h'_{n-1}$. Moreover $\hk_n$ separates $\hh_n$ from $\hh'_n$ and we have $\h_n \subsetneq \h_{n-1}  \subsetneq \h \subsetneq  \h'_{n-1} \subsetneq \h'_{n}$, as desired. 

This inductive construction yields an infinite sequence of  triples $(\h_n, \h'_n, \k_n)$ satisfying the conditions (a), (b), (c). 

Notice that $\hk_n$ crosses both $\hh_{n-1}$ and $\hh'_{n-1}$, it must in fact cross  $\hh_{m}$ and $\hh'_{m}$ for all $m<n$ by (c). In particular we have $\hk_n \neq \hk_m$ for all $m<n$ by (b). 
Moreover (b) also implies that $\hk_n$ must cross $\hk_m$ for all $m<n$. It follows that the hyperplanes $\hk_1, \hk_2, \dots$ are pairwise distinct and pairwise crossing. 
This contradicts the fact that $X$ is finite-dimensional (see Lemma~\ref{lem:dim}). 
\end{proof}
}

\section{Rank rigidity}\label{sec:RkOne}

The goal of this section is to provide a proof of Theorem~\ref{thmintro:RankRigidity} from the introduction. 

{
\subsection{Strongly separated hyperplanes and contracting isometries}

As mentioned in the introduction, the last missing piece in the proof of Rank Rigidity for \cat cube complexes is that a hyperbolic isometry which double skewers a pair of strongly separated hyperplanes must necessarily be contracting. This contracting behaviour will be deduced from the following key lemma, which readily implies that such an isometry must be rank one.

\begin{lem}\label{lem:KeyContracting}
Let $X$ be a finite-dimensional \cat cube complex, let $\h$ be a half-space and $\gamma \in \Aut(X)$ be a hyperbolic isometry with axis $\ell$ such that $\gamma \h \subsetneq \h$. Assume that the hyperplanes $\hh$ and $\gamma \hh$ are strongly separated. 

Then there is a constant $C$, depending only on $\gamma$, such that each geodesic segment crossing at least three walls in the orbit $\la \gamma \ra \hh$ has a non-empty intersection with the $C$-neighbourhood of $\ell$. 
\end{lem}

\begin{proof}
Let $p_0 \in \ell \cap \hh$ and set $\h_n = \gamma^n \h$ and $p_n = \gamma^n p_0$ for all $n \in \ZZ$. Let also  $N$ be the number of hyperplanes crossed by $[p_0, p_1]$. 

Let $a, b \in X$ be two points such that the geodesic segment $[a, b]$ crosses $\hh_{i-1}, \hh_i$ and $\hh_{i+1}$ for some $i \in \ZZ$. Let $x \in [a, b] \cap \hh_i$. We shall show that $x$ is at distance at most $N'$ away from $\ell$, where $N'$ is some constant depending only on $N$. 

By hypothesis, the two pairs of hyperplanes  $\hh_{i-1}, \hh_i$  and $ \hh_{i}, \hh_{i+1}$ are both strongly separated.  Therefore, none of the hyperplanes separating $p_i$ from $x$ can cross $\hh_{i-1}$ or $\hh_{i+1}$. Since $[a, b]$ is geodesic, it crosses each hyperplane at most once. It follows that each of the hyperplanes separating $p_i$ from $x$ must cross either $[p_{i-1}, p_i]$ or $[p_i, p_{i+1}]$. We conclude that the number of hyperplanes separating $p_i$ from $x$ is at most~$2N$. The desired conclusion follows since the \cat metric is quasi-isometric to the hyperplane distance (see Lemma~\ref{lem:cat_versus_combinatorial}). 
\end{proof}

\begin{lem}\label{lem:Contracting}
Let $X$ be a finite-dimensional \cat cube complex, let $\h' \subsetneq \h''$ be a nested pair of half-spaces and $\gamma \in \Aut(X)$ be such that $\h'' \subsetneq \gamma \h'$. 

If the hyperplanes $\hh', \hh''$ are strongly separated, then $\gamma$ is a contracting isometry. 
\end{lem}

\begin{proof}
Let $\ell$ be some $\gamma$-axis (notice that the hypotheses imply that  $\gamma$ is hyperbolic). Let also $p_0 \in \ell \cap \hh'$ and for all $n \in \ZZ$, set $p_n = \gamma^n p_0$ and $\h_n = \gamma^n \h'$. The hypotheses imply that for all $i \neq j \in \ZZ$, the hyperplanes $\hh_i, \hh_j$ are strongly separated. 

Suppose for a contradiction that $\gamma$ is not contracting. Then there are two sequences $(x_n)$ and $(y_n)$ in $X$ such that $d(x_n, y_n) \leq d(x_n, \ell)$ and that $\lim_n d(x'_n, y'_n) = \infty$, where $x'_n$ and $y'_n$ respectively denote the orthogonal projection of $x_n$ and $y_n$ on $\ell$. 

Since $[p_0, p_1]$ is a fundamental domain for the $\la \gamma \ra$-action on $\ell$, there is no loss of generality in assume that $x'_n \in [p_0, p_1]$ for all $n$. Upon extracting and reversing the orientation on $\ell$, we may further assume that  the $n$ hyperplanes $\hh_1, \dots, \hh_n$ separate $x'_n$ from $y'_n$ for all $n>0$.

Let $C$ be the constant from Lemma~\ref{lem:KeyContracting}. We claim that the number of hyperplanes in $\{ \hh_1, \dots, \hh_n\}$ that can be crossed by $[x_n, x'_n]$ or $[y_n, y'_n]$ is at most $C+3$. Indeed, since $x'_n$ is the orthogonal projection of $x_n$ on $\ell$, it follows that $d(z, \ell) = d(z, x'_n)$ for all $z \in [x_n, x'_n]$. Therefore, for $z \in [x_n, x'_n]$ with $d(z, x'_n)>C$, we deduce from Lemma~\ref{lem:KeyContracting} that $[z, x_n]$ cannot cross more than two walls in $\{ \hh_1, \dots, \hh_n\}$. This proves the claim for $[x_n, x'_n]$; the argument for  $[y_n, y'_n]$ is similar.

We next consider the geodesic quadrilateral with vertices $x_n, x'_n, y'_n, y_n$.  Each hyperplane separating $x'_n$ from $y'_n$ must cross one of the three  geodesic segments $[x_n, y_n]$, $[x_n, x'_n]$ or $[y_n, y'_n]$.  From the previous claim, we deduce that at least $n -2C -6$ of the hyperplanes $\hh_1, \dots, \hh_n$ must cross $[x_n, y_n]$. Let $i $ be the minimal index such that $[x_n, y_n]$ crosses $\hh_i$. Let $q_i \in [x_n, y_n ] \cap \hh_i$ and $q_{i+2} \in [x_n, y_n] \cap \hh_{i+2}$. By Lemma~\ref{lem:KeyContracting}, there is some $q \in [q_i, q_{i+2}]$ which is $C$-close to $\ell$. Thus we have
$$
\begin{array}{rcl}
d(x_n, y_n)  &\leq & d(x_n, x'_n)  \\
& =&   d(x_n, \ell)\\
& \leq & d(x_n, q) + d(q, \ell) \\
& \leq  &d(x_n, q) + C,
\end{array}
$$
whence $d(q, y_n) = d(x_n, y_n) - d(x_n, q) \leq C$. In particular, the number of hyperplanes crosses by $[q, y_n]$ is at most $C$. 

Recalling that the number of hyperplanes in  $\hh_1, \dots, \hh_n$  crossed by $[x_n, y_n]$ is at least $n-2C - 6$, we deduce that the number of hyperplanes crossed by $[q, y_n]$ must be least least  $n -2C - 8$. Thus we get a contradiction as soon as $n > 3C + 8$. 
\end{proof}
}

\subsection{Proof of Rank Rigidity}

In view of Proposition~\ref{prop:pruning}, Theorem~\ref{thmintro:RankRigidity} from the introduction is an immediate consequence of the following. 

\begin{thm}\label{thm:RankOne}
Let $X$ be a finite-dimensional \cat cube complex and $\Gamma \leq \Aut(X)$ be a group acting essentially without fixed point at infinity. 

Then $X$ is a product of two cube subcomplexes or every hyperplane of $X$ is skewered by a {contracting isometry} in $\Gamma$. 

If in addition $X$ is locally compact and $\Gamma$ acts cocompactly, then the same conclusion holds even if $\Gamma$ fixes a point at infinity. 
\end{thm}

\begin{proof}
{
Suppose that $X$ is irreducible and let $\h$ be a half-space. By Proposition~\ref{prop:IrredCriterion}, there is a pair of half-spaces $\h', \h''$ such that $\h' \subsetneq \h \subsetneq \h''$ and the hyperplanes $\hh'$ and $\hh''$ are strongly separated. From the Double Skewering Lemma, we deduce that there is some  $\gamma \in \Gamma$ such that $\h'' \subsetneq \gamma \h'$. In particular $\gamma $ skewers $\hh$. We conclude by invoking Lemma~\ref{lem:Contracting}, which ensures that $\gamma$ is a contracting isometry. }

\medskip
Assume now that $X$ is locally compact and that $\Gamma$ acts cocompactly (but not necessarily properly), and assume that $X$ is irreducible. We then invoke Corollary~\ref{cor:Flipping:cocpt}. Two situations can occur. The first is that all hyperplanes are compact; in that case {the Flat Plane Theorem (see \cite[Theorem~III.H.1.5]{Bridson-Haefliger})  implies that  $X$ is Gromov hyperbolic, and  it is then clear that all hyperbolic isometries are contracting}. The second is that $\Gamma$ has no fixed point at infinity, and we are then reduced to a situation which has already been dealt with. 
\end{proof}

\begin{cor}\label{cor:RkRigid}
Let $X$ be an unbounded locally compact \cat cube complex  such that $\Aut(X)$ acts cocompactly and let $\Gamma \leq \Aut(X)$ be a {possibly non-uniform} lattice. We have the following.
\begin{enumerate}[(i)]
\item $\Gamma$ acts essentially on the essential core $Y$ of $X$. 

\item $Y$ embeds as an $\Aut(X)$-invariant convex subcomplex in $X$.

\item $Y$ decomposes as a product $Y \cong Y_1 \times \dots \times Y_m$ of $m\geq 1$ irreducible unbounded convex subcomplexes, and for each $i$, there is some  $\gamma_i \in \Gamma$ preserving that decomposition and acting on $Y_i$ as  a rank one isometry. 
\end{enumerate}
\end{cor}

\begin{proof}
{We remark that, since $\Gamma$ is possibly non-uniform, its action on $X$ is not necessarily cocompact. The goal is thus to reduce to a situation where $\Gamma$ acts essentially without a fixed point at infinity, so that the first part of Theorem~\ref{thm:RankOne} will give the desired conclusion.}

We start by applying Proposition~\ref{prop:pruning} to the whole group $\Aut(X)$ (see also Proposition~\ref{prop:essential:bis}). This shows that there is thus no loss of generality in assuming that $\Aut(X)$ acts essentially on $X$. By Proposition~\ref{prop:deRham} we have a canonical product decomposition $X \cong X_1 \times \dots \times X_m$, where $X_i$ is an irreducible subcomplex. Moreover, upon replacing $\Gamma$ by a finite index subgroup, we can assume that $\Gamma$ preserves this decomposition. 

For all $i = 1, \dots, m$, we need to show that $\Gamma$ acts essentially on $X_i$ and that it contains an element $\gamma_i$ acting as a rank one isometry on $X_i$. We already know from Theorem~\ref{thm:RankOne} that $\Aut(X_i)$ contains rank one isometries. 

Since $\Aut(X_i)$ acts cocompactly, there is a non-empty minimal closed convex \cat subspace $Z_i \subseteq X_i$, which need not be a subcomplex. Since $\Aut(X_i)$ contains rank one elements, it follows that $Z_i$ is an irreducible \cat space. If $Z_i$ is flat, then it is isometric to the real line and $X$ is $\RR$-like. The desired result is then clear. We assume henceforth that $Z_i$ is not flat. 

We claim  that $\Aut(X_i)$ has no fixed point in $\bd X_i$. This follows from \cite{CapraceMonod3} (for a special case, see also Theorem~3.14 from \cite{CapraceMonod2}); a direct argument in the current specific setting could also be obtained using Corollary~\ref{cor:Flipping:cocpt}. 

Now, from the `geometric Borel density theorem' proved in  \cite[Theorem~2.4]{CapraceMonod2}, we now infer that $\Gamma$ has no fixed point at infinity of $Z_i$ and that it acts minimally on $Z_i$, \emph{i.e.} $Z_i$ contains no non-empty proper $\Gamma$-invariant closed \cat subspace. Since every hyperplane of $X_i$ separates $Z_i$, it follows in particular that for any $z \in Z_i$, the orbit $\Gamma.z$ contains points on both sides of every hyperplane. Recall that each hyperplane is $\Aut(X_i)$-essential. In particular, given a half-space $\h$, there is a half-space $\k \subset \h$ such that the distance from any point of $\k$ to $\hh$ is arbitrarily large. Since $\k$ intersects non-trivially  the orbit $\Gamma.z$, we deduce that $\h$ contains points of the orbit $\Gamma.z$ which are arbitrarily far away from the hyperplane $\hh$. In other words, this means that $\h$ is $\Gamma$-deep. Since $\h$ was arbitrary, this proves that every hyperplane of $X$ is $\Gamma$-essential or, equivalently, that the $\Gamma$-action on $X_i$ is essential. 

We have seen that  $\Gamma$ has no fixed point at infinity of $Z_i$. Since $Z_i \subseteq X_i$ is $\Aut(X_i)$-invariant and since $\Aut(X_i)$ acts cocompactly on $X_i$, it follows that $\Gamma$ has no fixed point at infinity of $X_i$. At this point, we are able to invoke Theorem~\ref{thm:RankOne}, which provides an element $\gamma_i$ in $\Gamma$ acting as a rank one isometry on $X_i$. 
\end{proof}

\begin{cor}\label{cor:HereditarilyEss}
Let $X$ be a locally compact \cat cube complex such that $\Aut(X)$ acts cocompactly and essentially. Then the action of any {(possibly non-uniform)} lattice $\Gamma \leq \Aut(X)$ is hereditarily essential. 
\end{cor}

\begin{proof}
For each hyperplane $\hk$, the stabilizer $\Stab_{\Aut(X)}(\hk)$ is an open subgroup of $\Aut(X)$ which acts cocompactly on $\hk$. It follows that $\Stab_\Gamma(\hk) = \Gamma \cap \Stab_{\Aut(X)}(\hk)$ is a lattice in $ \Stab_{\Aut(X)}(\hk)$ to which Corollary~\ref{cor:RkRigid} applies. This shows in particular that $\Stab_\Gamma(\hk)$ acts essentially on the essential core of $\hk$, and the desired result follows by induction on dimension. 
\end{proof}

We can now complete the proof of Corollary~\ref{cor:GeodComplete} from the introduction, which is concerned with the geodesically complete case. 

\begin{proof}[Proof of Corollary~\ref{cor:GeodComplete}]
A group acting cocompactly on a geodesically complete \cat space necessarily acts minimally, see \cite[Lemma~3.13]{CapraceMonod1}. In view of Lemma~\ref{lem:HalfEssential} and Remark~\ref{rem:HessEmpty}, this implies that every hyperplane is $\Aut(X)$-essential. The result then follows from Corollary~\ref{cor:RkRigid}.
\end{proof}

\section{Applications}

\subsection{Cube complexes with invariant Euclidean flats}

We start with an elementary observation (see \S\ref{sec:R-like} for the definition of an $\RR$-like complex).

\begin{lem}\label{lem:pseudoEucl}
Let $X$ be a finite-dimensional \cat cube complex such that $\Aut(X)$ acts essentially. Then $\Aut(X)$ stabilizes some $n$-dimensional flat $\RR^n \subseteq X$ if and only if $X$ decomposes as a product $X = X_1 \times \dots \times X_n$ of subcomplexes, each of which is essential and $\RR$-like. 
\end{lem}

\begin{proof}
Clearly, an essential \cat cube complex $X_i$ that is $\RR$-like is irreducible. Therefore, a product $X= X_1 \times \dots \times X_n$ of $n$ complexes of this form has the property that $\Aut(X)$ preserves some Euclidean flat, since this product decomposition is preserved by some finite index subgroup of $\Aut(X)$ (see Proposition~\ref{prop:deRham}). 

Assume conversely that   $\Aut(X)$ stabilizes some $n$-dimensional flat $F \subseteq X$. Since $X$ is essential, every hyperplane must separate $F$ into two non-empty disjoint pieces. More precisely, for each $\hh \in \hH(X)$, the intersection $\hh \cap F$ is a Euclidean hyperplane of $F$. Let $\hH_1$ be the collection of all hyperplanes $\hk \in \hH(X)$ such that $\hk \cap F$ is parallel to $\hh \cap F$ in the sense of Euclidean geometry. Let also $\hH_2 = \hH(X) - \hH_1$. If $X$ is not $\RR$-like, that is to say if $n>1$, then $\hh \cap F$ is non-compact and it follows that $\hH_2$ is not empty. Lemma~\ref{lem:Prod} then yields a product decomposition of $X$ into two subcomplexes and the result follows by induction on dimension. 
\end{proof}

\begin{thm}\label{thm:PseudoEucl}
Let $X$ be a finite-dimensional  \cat cube complex such that $\Aut(X)$ acts essentially and satisfies at least one of the following conditions:
\begin{enumerate}[(a)]
\item $\Aut(X)$ has no fixed point at infinity. 

\item $\Aut(X)$ acts cocompactly and $X$ is locally compact.
\end{enumerate}
Then {$\Aut(X)$ stabilises some Euclidean flat  if and only if there is no facing triple of hyperplanes. }
\end{thm}


\begin{proof}
Assume first that $X$ contains  an $\Aut(X)$-invariant flat $F $ (possibly reduced to a singleton). Since every hyperplane is $\Aut(X)$-essential, every hyperplane separates $F$ into two non-empty  disjoint convex pieces. Therefore, it follows from the definition of a hyperplane that for each $\hh \in \hH(X)$, the intersection $\hh \cap F$ is a Euclidean hyperplane of $F$. The desired assertion follows since there is no facing triple of hyperplanes in Euclidean geometry. 

Assume now that $X$ does not contain any $\Aut(X)$-invariant flat $F $; in particular it is unbounded. Let $X= X_1 \times \dots \times X_n$ be the canonical decomposition provided by Proposition~\ref{prop:deRham}. It suffices to prove that one of the factors $X_i$ contains a facing triple of hyperplanes. By Lemma~\ref{lem:pseudoEucl},  one of the irreducible factors, say $X_i$, is not $\RR$-like. Thus there is no loss of generality in assuming that $X = X_i$ or, equivalently, that $X$ is irreducible but not $\RR$-like. 

If condition (b) holds, then Corollary~\ref{cor:Flipping:cocpt} (applied with $G = \Aut(X)$) implies that $\Aut(X)$ has no fixed point at infinity since $X=X_i$ is irreducible. Thus it suffices to prove the existence of a facing triple under the assumption that condition (a) holds. 

Assume first that each hyperplane of $X$ is compact. In that case, the existence of a facing  triple of hyperplanes is easy to establish; details are left to the reader. We assume henceforth  that some  hyperplane, say $\hh$, is non-compact.

By Proposition~\ref{prop:IrredCriterion}, there is a pair of strongly separated hyperplanes $\hh', \hh''$ such that $\h' \subsetneq \h \subsetneq \h''$. By the Double Skewering Lemma, there is some $g \in \Aut(X)$ such that $\h'' \subsetneq g\h'$. In particular, each pair of hyperplanes  in the $\la g \ra$-orbit of $\hh$ is strongly separated. 

Since $\hh$ is not compact, infinitely many hyperplanes cross it. Amongst them, none can cross 
$g\inv \hh$ or $g \hh$, and only finitely many separate $g\inv \hh$ from $g \hh$. It follows that there is some halfspace $\k$ containing both $g\inv \hh$ and $g \hh$, and such that $\hk$ crosses $\hh$. It follows that $\hk$, $g\inv \hh$ and $g \hh$ forms a facing triple of hyperplanes.
%
\end{proof}

\subsection{Tits alternative, second version}

\begin{proof}[Proof of Corollary~\ref{cor:TitsAlt}]
It follows from \cite[Theorem~1.7]{CapraceLytchak} that $\Gamma$ is amenable if and only if it is \{locally finite\}-by-\{virtually Abelian\}. Thus it suffices to show that if $\Gamma$ is not amenable, then it contains a non-Abelian free subgroup. We assume henceforth that $\Gamma$ is non-amenable. From \cite[Theorem~A.5]{CapraceLecureux}, we deduce that $\Gamma$ does not fix any point in the \emph{ultrafilter bordification} of $X$ (see the Appendix of \cite{CapraceLecureux} for details). In particular, this means that there is some $\Gamma$-invariant restriction quotient $Y$ of $X$ such that the $\Gamma$-action on $Y$ has no fixed point in $Y \cup \bd Y$. Theorem~\ref{thm:TitsAlt} then ensures that $\Gamma$ contains a free subgroup. 
\end{proof}

\subsection{Regular elements}

We first need to recapitulate the information we have obtained so far on lattices of locally compact \cat cube complexes.

\begin{prop}\label{prop:recap:lattices}
Let $X$ be a locally compact \cat cube complex such that $\Aut(X)$ acts essentially and cocompactly and let $\Gamma \leq \Aut(X)$ be a {(possibly non-uniform)} lattice. Then:
\begin{enumerate}[(i)]
\item $X$ decomposes as a product $X = X_1 \times \dots \times X_p \times Y$ of subcomplexes such that $X_i$ is $\RR$-like for each $i$ and $Y$ has no $\RR$-like factor.

\item Every automorphism preserves the decomposition upon permuting some possibly isomorphic factors $X_i$. 

\item The $\Gamma$-action {on $X$ (and hence on $Y$)} is hereditarily essential.

\item  $\Gamma$ does not fix any point in $\bd Y$. 

\item Every pair of disjoint hyperplanes is double-skewered by some element of $\Gamma$. 

\end{enumerate}
\end{prop}

\begin{proof}
First notice that, by Proposition~\ref{prop:deRham}, there is a canonical decomposition $X = X_1 \times \dots \times X_k$ into irreducible subcomplexes, which is virtually invariant under $\Aut(X)$. Upon reordering, we can assume that the $\RR$-like factors are precisely the first $p$ ones. In particular every automorphism permutes the first  $p$ factors among themselves. 

Next consider any point $x = (x_1, \dots, x_p) \in X_1 \times \dots \times X_p$. Since {the stabiliser $\Aut(X)_x$ of $x$ in $\Aut(X)$} is open, it follows that $\Gamma_x$ is a lattice in $\Aut(X)_x$ and therefore acts as a lattice on $Y = X_{p+1} \times \dots \times X_k$. We then observe that the $\Gamma$-action on $X$, as well as the $\Gamma_x$-action on $Y$, is hereditarily essential by Corollary~\ref{cor:HereditarilyEss}. We can thus invoke Corollary~\ref{cor:Flipping:cocpt}. Since $Y$ does not have any $\RR$-like factor, this implies in particular that $\Gamma_x$ does not fix any point in $\bd Y$. Thus (iv) holds. 

It only remains to prove (v). To this end, remark that the splitting $X = X_1 \times \dots \times X_p \times Y$ comes with a decomposition of the set of hyperplanes $\hH(X) = \hH(X_1) \cup \dots \cup \hH(X_p) \cup \hH(Y)$. Clearly any pair of disjoint hyperplanes must belong to the same component. If they belong to $\hH(Y)$, then we conclude by using any of the two versions of the Double Skewering Lemma. If they belong to $\hH(X_i)$, we just remark that any hyperbolic isometry of $X_i$ will double-skewer the given pair of hyperplanes as desired, since $X_i$ is $\RR$-like and all its hyperplanes are compact. 
\end{proof}

\begin{proof}[Proof of Theorem~\ref{thm:Regular}]
Let $\hH_i = \hH(X_i)$ for all $i$, so that the given splitting $X = X_1 \times \dots \times X_n$ yields a decomposition $\hH(X) = \hH_1 \cup \dots \cup \hH_n$. 

For each $i$, pick two hyperplanes $\hh_i, \hk_i \in \hH_i$ such that $\h_i \subset \k_i$ and that any element $g$ of $\Aut(X_i)$ with $g.\k_i \subset \h_i$ is rank one. The existence of such a pair of hyperplanes is guaranteed by Proposition~\ref{prop:IrredCriterion} and Lemma~\ref{lem:Contracting}. In particular, it suffices to exhibit an element $g \in \Gamma$ such that $g.\k_i \subset \h_i$ simultaneously for all $i$.  We shall do this by induction on $n$, the case $n=1$ being covered by Proposition~\ref{prop:recap:lattices}.


Let $x$ be any point of $X_n$ contained in the projection of $\hk_n$ and consider its stabilisers $\Gamma_x$ and $\Aut(X)_x$. Since $X_n$ is locally compact, there are finitely many hyperplanes in $X_n$ containing the point $x$. Therefore $\Gamma_x$ has a finite index subgroup $\Gamma'_x$ which stabilizes $\hk_n$. Upon replacing $\Gamma'_x$ by a subgroup of index two, we may further assume that $\Gamma'_x$ stabilizes both $\k_n$ and $\k_n^*$.  Moreover, since {the stabiliser $\Aut(X)_x$ of $x$ in $\Aut(X)$} is an open subgroup of $\Aut(X)$, it follows that $\Gamma_x = \Gamma \cap \Aut(X)_x$ is a lattice in $\Aut(X)_x$, and so is thus $\Gamma'_x$. The group $\Aut(X_n)_x$ being compact, it follows that the image of the projection of $\Gamma'_x$ to $\Aut(X_1) \times \dots \times \Aut(X_{n-1})$ is a lattice.  For all $i$, since $\hk_i$ and $\hh_i$ are essential hyperplanes of $X$, it follows that they are also essential as hyperplanes of $X_i$. In particular, they are $\Gamma'_v$-essential for all $i < n$ (see Corollary~\ref{cor:HereditarilyEss}), and the induction hypothesis then yields an element $a \in \Gamma'_x$ such that $a.\k_i \subset \h_i$  for all $i<n$ and $a.\k_n = \k_n$. 

We now pick a point $y \in X_1 \times \dots \times X_{n-1}$ contained in the projection of $a.(\hk_1 \cap \dots \cap \hk_{n-1})$. For the same reason as before, we can use the induction hypothesis to find an element $b \in \Gamma_y$ which stabilizes $a.\k_i$ for all $i<n$  and such that $b.\k_n \subset \h_n$. 

It remains to set $g = ba$ and observe that $g.\k_i \subset \h_i$ for all $i$, as required.
\end{proof}

\begin{proof}[Proof of Corollary~\ref{cor:Zn}]
There is no loss of generality in assuming that $\Gamma$ acts essentially (see Corollary~\ref{cor:RkRigid}) and that $X$ decomposes as a product $X = X_1 \times \dots \times X_n$ of $n$ irreducible subcomplexes. In particular $\Aut(X)$ is virtually isomorphic to $\Aut(X_1) \times \dots \times \Aut(X_n)$ by Proposition~\ref{prop:deRham}. 

Let $\gamma \in \Gamma$ be a regular element as provided by Theorem~\ref{thm:Regular}. For each irreducible factor $X_i$ of $X$, the centraliser in $\Aut(X_i)$  of the projection of $\gamma$ stabilises the pair consisting of its attracting and repelling fixed points at infinity. In particular this centraliser stabilises some geodesic line in $X_i$. It follows that the centraliser of $\gamma$ in $\Aut(X)$ stabilises some $n$-flat embedded in $X$ on which it acts cocompactly. 

By a Lemma of Selberg \cite{Selberg}, the centraliser $\centra_\Gamma(\gamma)$ is a cocompact lattice in the centraliser $\centra_{\Aut(X)}(\gamma)$. We have just seen that, up to some compact normal subgroup, the latter is isomorphic to closed subgroup of $\Isom(\RR^n)$. Thus $\centra_\Gamma(\gamma)$ is a Bieberbach group of rank $n$, and thus contains $\ZZ^n$. 
\end{proof}

\subsection{Quasi-morphisms}

Theorem~\ref{thm:ProductTrees} from the introduction will be deduced from the following.

\begin{thm}\label{thm:QH}
Let $X = X_1 \times \dots \times X_n$ be a product of $n$ geodesically complete locally compact \cat spaces. Assume that $\Isom(X_i)$ acts cocompactly on $X_i$ and contains some rank one isometry. For any lattice $\Gamma \leq \Isom(X_1) \times \dots \times \Isom(X_n)$, the following conditions are equivalent.
\begin{enumerate}[(i)]
\item $\QH(\Gamma) $ is finite-dimensional. 

\item $\QH(\Gamma) = 0$. 

\item For all $i$, the space $X_i$ is either a tree or a rank one symmetric space, and if $X_i$ is not isometric to the real line $\RR$, then the closure $G_i$ of the projection of $\Gamma$ on $\Isom(X_i)$ acts doubly transitively on $\bd X_i$.
\end{enumerate}
\end{thm}

The following basic facts shows that the set of rank one elements is `large' as soon as it is non-empty. 

\begin{lem}\label{lem:Ballmann:bis}
Let $X$ be a proper \cat space such that $\Isom(X)$ contains a rank one isometry. Then any group $\Gamma \leq \Isom(X)$ acting without fixed point at infinity and with full limit set also contains a rank one isometry. 
\end{lem}

\begin{proof}
Let $g \in \Isom(X)$ be a rank one isometry. Let also $\xi_+, \xi_- \in \bd X$ be the attracting and repelling fixed points of $g$ respectively.  Since $\Gamma$ has full limit set, there is a sequence $\gamma_n \in \Gamma$ such that for some (and hence for all) point $p \in X$, we have $\lim_{n \to \infty} \gamma_n.p = \xi_+$. Upon extracting a subsequence, we can assume moreover that the sequence $ \gamma\inv_n.p $ also converges as $n $ tends to infinity, say, to some point $\eta \in \bd X$. (Clearly $(\gamma\inv_n.p)$ is unbounded, otherwise $(\gamma_n.p)$ would be bounded.)

If $\eta = \xi_+$, then we pick any element $\lambda \in \Gamma$ such that $\lambda.\xi_+ \neq \xi_+$ and we set $\gamma'_n = \gamma_n\lambda$ for all $n$. In that case we still have  $\lim_{n \to \infty} \gamma'_n.p = \xi_+$, and we have also 
$$\lim_{n \to \infty} (\gamma'_n)\inv.p = \lambda\inv.(\lim_{n \to \infty} \gamma_n\inv.p) = \lambda\inv.\xi_+ \neq \xi_+.$$
It then follows from Lemmas~III.3.1 and III.3.2 from \cite{Ballmann} that
 $\gamma'_n$ is rank one for all sufficiently large $n>0$. 
\end{proof}

We shall also use the following trichotomy, which is obtained in \cite{CapraceFujiwara} as a consequence of the construction of quasi-morphisms by M.~Bestvina and K.~Fujiwara in \cite{BestvinaFujiwara}.

\begin{prop}[Theorem~1.8 from \cite{CapraceFujiwara}]\label{prop:QH:trichotomy}
Let $X$ be a proper \cat space with cocompact isometry group, and let $G \leq \Isom(X)$ be a group containing some rank one isometry. Then one of the following assertions hold. 
\begin{enumerate}[(i)]
\item $G$ fixes a point in $\bd X$ or stabilises some geodesic line, and the closure $\overline G \leq \Isom(X)$ is amenable. 

\item $\overline G$ acts doubly transitively on $\bd X$ and the space $\QH_c(\overline G)$ of continuous non-trivial quasi-morphisms vanishes. 

\item $\overline G$ is not doubly transitive on $\bd X$, and the spaces $\QH(G)$ and $\QH_c(\overline G)$ are both infinite-dimensional.
\end{enumerate}
\end{prop}

A final ingredient need for the proof of Theorem~\ref{thm:QH} is the following result, which is a consequence of \cite[Theorem~1.3]{CapraceMonod1}. 

\begin{prop}\label{prop:DoublyTrans}
Let $X$ be a locally compact geodesically complete \cat space with cocompact isometry group and assume that $X$ is not reduced to a single point. 

Then $\Isom(X)$ is doubly transitive on $\bd X$ if and only if $X$ is isometric either to a rank one symmetric space or to a locally finite semi-regular tree. 
\end{prop}

\begin{proof}
The `if' part is clear. We assume henceforth that  $\Isom(X)$ is doubly transitive on $\bd X$. This hypothesis implies in particular that any two points of the boundary can be joined by a geodesic line. In particular any two points of the boundary are at distance $\pi$ in the angular metric. Hence $X$ is a visibility space, and it follows from Theorems~II.9.33 and~III.1.5 in \cite{Bridson-Haefliger} that $X$ is Gromov hyperbolic.

Fix a point $\xi \in \bd X$ and pick any base point $p \in X$. The geodesic ray $[p, \xi)$ can be extended to some geodesic line  $\ell$. Notice that the set $P(\ell)$ consisting the union of all geodesic lines at bounded Hausdorff distance from $\ell$ is contained in bounded neighbourhood of $\ell$, since otherwise $\bd X$ would contain some point at distance $\pi/2$ from $\xi$. 

For any other point $x \in X$, there is a geodesic line $\ell'$ containing $x$ and having $\xi$ as one of its endpoints. Therefore there is an element of $\Isom(X)_\xi$ which maps $\ell'$ to a line parallel to $\ell$. This shows that  $X = \Isom(X)_\xi.P(\ell)$. Notice that $\Isom(X)$ contains some hyperbolic isometry; this follows \emph{e.g.} from the fact that $X$ is Gromov hyperbolic and unbounded. Since $\Isom(X)$ is doubly transitive on the boundary, it follows in particular that $\Isom(X)_\xi$ contains some element acting cocompactly on $P(\ell)$. Combining this with the fact that   $X = \Isom(X)_\xi.P(\ell)$, we infer that $\Isom(X)_\xi$ acts cocompactly on $X$. The desired result then follows from  \cite[Theorem~1.3]{CapraceMonod1}. 
\end{proof}

\begin{proof}[Proof of Theorem~\ref{thm:QH}]
Since $\Isom(X_i)$ acts cocompactly on $X_i$, it has full limit set. From Proposition~2.9 in \cite{CapraceMonod2}, we deduce that the projection $\Gamma_i$ of $\Gamma$ to $\Isom(X_i)$ also has full limit set. Lemma~\ref{lem:Ballmann:bis} thus ensures that $\Gamma_i$ contains some rank one isometry. We are now ready to show the desired equivalences. 

\begin{compactenum}
\item[(ii) $\Rightarrow$ (i)] Obvious. 

\item[(iii) $\Rightarrow$ (ii)] In case each $X_i$ is a tree, the vanishing of $\QH(\Gamma)$ is explicitly stated as Corollary~26 in \cite{BurgerMonod}, modulo the fact that $\QH(\Gamma)$ coincides with the kernel of the canonical map from the bounded cohomology of $\Gamma$ with trivial coefficients to the usual cohomology in degree~$2$ (see \emph{e.g.} \cite[Cor.~13.3.2]{Monod}). In case some $X_i$ is a rank one symmetric space, the same proof applies \emph{verbatim}. Thus  $\QH(\Gamma) = 0$ as desired. 

\item[(i) $\Rightarrow$ (iii)] Assume that (iii) fails. In view of Proposition~\ref{prop:DoublyTrans}, this amounts to saying that for some $i$, the space $X_i$ is not isometric to the real line and that the closure $G_i$ of $\Gamma_i$ does not act doubly transitively on $\bd X_i$. By the trichotomy of Proposition~\ref{prop:QH:trichotomy}, it follows that either $G_i$ stabilises some geodesic line in $X_i$ or that $\QH(\Gamma)$ is infinite-dimensional. But the former possibility implies that $X_i$ is isometric to the real line, since any group acting cocompactly on a geodesically complete \cat space does not preserve any non-empty proper subspace (see \cite[Lemma~3.13]{CapraceMonod1}), and is thus precluded. Hence (i) fails as well, as was to be shown. \qedhere
\end{compactenum}
\end{proof}

\begin{remark}\label{rem:QH}
A similar result holds without the hypothesis that $X_i$ is geodesically complete. In that case, the condition (iii) must be weakened and replaced by the fact that $X_i$ is quasi-isometric to a rank one symmetric space or a semi-regular tree, and the $\Gamma$-action on $X_i$ is \textbf{quasi-distance-transitive} in the sense that there is some $\delta >0$ such that for all $x, y,x', y' \in X_i$, there is some $\gamma \in \Gamma$ such that $d_{X_i}(x, \gamma.x') < \delta$ and $d_{X_i}(y, \gamma.y')< \delta$.  This can be established by following the same proof; the main change is that Proposition~\ref{prop:DoublyTrans} must be replaced by the `quasi-isometric' version we have just described. 
\end{remark}

\begin{proof}[Proof of Theorem~\ref{thm:ProductTrees}]
As in the proof of \ref{thm:Regular}, we remark that $\Aut(X_i)$ contains a lattice for each $i$. By Corollary~\ref{cor:GeodComplete}, it follows that $\Aut(X_i)$ contains some rank one isometry. The desired result then follows from Theorem~\ref{thm:QH}.
\end{proof}

\subsection{Asymptotic cones}

The existence of cut-points in asymptotic cones of a space or a group has been studied notably in \cite{DMS}. A function $\mathrm{Div} : \NN \to \RR \cup \{\infty\}$, called the \textbf{divergence},  is associated to any path-connected metric $X$. In case $X$ is proper and has a cocompactly isometry group, it is shown in \cite[Prop.~1.1]{DMS} that $\mathrm{Div}$ grows linearly with $n$ if and only if no asymptotic cone of $X$ has a cut-point. 

\begin{proof}[Proof of Corollary~\ref{cor:AsymCones}]
{We do not provide all the details; the reader is assumed to have some familiarity with \cite{DMS}.}

Since $X$ is essential (see Proposition~\ref{prop:essential:bis}), it follows that any irreducible factor of $X$ is non-compact. In particular, if $X$ is not irreducible, then it follows that it has linear divergence and, hence, no asymptotic cone has a cut-point. On the other hand, if $X$ is irreducible, then Theorem~\ref{thmintro:RankRigidity} ensures that $\Aut(X)$ contains a rank one element. In particular $X$ contains some rank one geodesic. By Proposition~3.24 from \cite{DMS}, this implies that every asymptotic cone of $X$ has a cut-point. 

We now turn to the subgroup $\Gamma$.  Theorem~\ref{thmintro:RankRigidity} again implies that $\Gamma$ contains some rank one element. By Proposition~3.26 from \cite{DMS}, this implies that $\Gamma$ contains Morse elements (as defined in \cite{DMS}) relatively to its own word metric (even if $\Gamma$ is distorted in $X$), and this in turn implies the existence of asymptotic cut-points by Proposition~3.24 from \emph{loc.~cit.} 
\end{proof}

\providecommand{\bysame}{\leavevmode\hbox to3em{\hrulefill}\thinspace}
\providecommand{\MR}{\relax\ifhmode\unskip\space\fi MR }
\providecommand{\MRhref}[2]{%
  \href{http://www.ams.org/mathscinet-getitem?mr=#1}{#2}
}
\providecommand{\href}[2]{#2}

\end{document}